\documentclass[11pt]{amsart}
\usepackage[a4paper,left=3cm,right=3cm]{geometry}
\usepackage[latin1]{inputenc}
\usepackage[english]{babel}
\usepackage{amsmath}
\usepackage{verbatim}
\usepackage{amssymb}
\usepackage{color}
\usepackage{hyperref}
\usepackage{amsthm}
\usepackage{tikz}
\usepackage{pgfplots}
\usepackage{tikz-3dplot}
\usepackage{float}
\usepackage{amsfonts} 
\usepackage{enumerate}

\newtheorem{theorem}{Theorem}[section]
\newtheorem{conjecture}{Conjecture}

\newtheorem{lemma}[theorem]{Lemma}
\newtheorem{proposition}[theorem]{Proposition}
\newtheorem{remark}[theorem]{Remark}

\newcommand{\R}{\mathbb R}

\begin{document}
\title{A comparison between Neumann and Steklov eigenvalues}
\author[]{Antoine Henrot, Marco Michetti}

\address[Antoine Henrot]{Universit\'e de Lorraine, CNRS, IECL, F-54000 Nancy, France}
\email{antoine.henrot@univ-lorraine.fr}
\address[Marco Michetti]{
Universit\'e de Lorraine, CNRS, IECL, F-54000 Nancy France}
\email{marco.michetti@univ-lorraine.fr}

\date{\today}

\begin{abstract}
    This paper is devoted to a comparison between the normalized first (non-trivial) Neumann eigenvalue $|\Omega| \mu_1(\Omega)$
    for a Lipschitz open set $\Omega$ in the plane, and  the normalized first (non-trivial) Steklov eigenvalue $P(\Omega) \sigma_1(\Omega)$. 
    More precisely, we study the ratio $F(\Omega):=|\Omega| \mu_1(\Omega)/P(\Omega) \sigma_1(\Omega)$. We prove that this ratio can
    take arbitrarily small or large values if we do not put any restriction on the class of sets $\Omega$. Then we restrict ourselves to the class
    of plane convex domains for which we get explicit bounds. We also study the case of thin convex domains for which we give more precise
    bounds. The paper finishes with the plot of the corresponding Blaschke-Santal\'o diagrams $(x,y)=\left(|\Omega| \mu_1(\Omega), P(\Omega)
     \sigma_1(\Omega) \right)$.
\end{abstract}

\maketitle

\tableofcontents

\section{Introduction}

Let $\Omega\subset \mathbb{R}^2$ be an open Lipschitz set, the Steklov problem on $\Omega$ consists in solving the eigenvalue problem 
\begin{equation*}
\begin{cases}
     \Delta v=0\qquad  \Omega   \\
      \partial_{\nu} v=\sigma v \qquad \partial\Omega,
\end{cases}
\end{equation*}
where $\nu$ stands for the outward normal at the boundary.
As the trace operator $H^1(\Omega) \rightarrow L^2(\partial \Omega)$ is compact (when $\Omega$ is Lipschitz), the spectrum of the Steklov problem is discrete and the eigenvalues (counted with their multiplicities) go to infinity
$$0= \sigma_0(\Omega) \leq  \sigma_1(\Omega)\le  \sigma_2 (\Omega) \le \cdots \rightarrow +\infty.$$
We recall the classical variational characterization of the Steklov eigenvalues
\begin{equation}\label{eVFS}
\sigma_k(\Omega)=\sup_{E_k} \inf_{0\neq v\in E_k} \frac{\int_{\Omega}|\nabla v|^2dx}{\int_{\partial \Omega}v^2 ds},
\end{equation}
where the infimum is taken over all $k-$dimensional subspaces of the Sobolev space $H^1(\Omega)$ which are $L^2-$orthogonal to constants on $\partial \Omega$.

The Neumann eigenvalue problem on $\Omega$ consists in solving the eigenvalue problem 
\begin{equation*}
\begin{cases}
     -\Delta u=\mu u\qquad  &\Omega \\
      \partial_{\nu} u=0 \qquad &\partial\Omega.
\end{cases}
\end{equation*}
As the Sobolev embedding $H^1(\Omega) \rightarrow L^2(\Omega)$ is also compact here, the spectrum of the Neumann problem is discrete and the eigenvalues (counted with their multiplicities) go to infinity
$$0= \mu_0(\Omega) \leq   \mu_1(\Omega)\le  \mu_2 (\Omega) \le \cdots \rightarrow +\infty.$$
We also have a variational characterization of the Neumann eigenvalues
\begin{equation}\label{eVFN}
\mu_k(\Omega)=\sup_{E_k} \inf_{0\neq u\in E_k} \frac{\int_{\Omega}|\nabla u|^2dx}{\int_{ \Omega}u^2 dx},
\end{equation}
where the infimum is taken over all $k-$dimensional subspaces of the Sobolev space $H^1(\Omega)$ which are $L^2-$orthogonal to constants on $\Omega$.

Recently several papers study the link between theses two families of eigenvalues, let us mention for example 
 \cite{GHL20}, \cite{GKL20}, \cite{HS20}, \cite{LP15}.
A natural question is to compare the first (non-trivial) eigenvalues suitably normalized, that is to say to compare $|\Omega|\mu_1(\Omega)$ and
$P(\Omega)\sigma_1(\Omega)$
where $\Omega\subset \mathbb{R}^2$ is an open Lipschitz set in the plane, $|\Omega|$ is its Lebesgue measure, $P(\Omega)$ is its perimeter.
More precisely, in this paper we study the following spectral shape functional:
\begin{equation}\label{eF1}
F(\Omega)=\frac{\mu_1(\Omega)|\Omega|}{\sigma_1(\Omega)P(\Omega)}.
\end{equation}
We want to find bounds for $F(\Omega)$ (if possible optimal) in the two following cases: the set $\Omega\subset \R^2$ is just bounded and Lipschitz
or the set $\Omega\subset \R^2$ is bounded and convex.

We now present the main results and the structure of the paper. In Section \ref{sE} we will show that, if we do not put any restriction on the class of sets, the problem of maximization and minimization of $F(\Omega)$ is ill posed, indeed we have 
\begin{equation*}
\inf\{F(\Omega):\Omega\subset \mathbb{R}^2 \text{ bounded open set and Lipschitz} \}=0,
\end{equation*}
\begin{equation*}
\sup\{F(\Omega):\Omega\subset \mathbb{R}^2 \text{ bounded open set and Lipschitz} \}=+\infty.
\end{equation*}
Thus we will study the problem of minimizing or maximizing $F(\Omega)$ in the class of convex plane domains.  It is well known that minimizing (or maximizing)
sequences of plane convex domains 
\begin{itemize}
\item either converge (in the Hausdorff sense) to an open convex set and we will see that, in this case, this set will be the minimizer or maximizer,
\item or shrink to a segment which leads us to consider such particular sequences of convex domains.
\end{itemize} 

Therefore, in Section \ref{sCS} we will study the behaviour of the functional $F(\Omega_{\epsilon})$ where $\Omega_{\epsilon}$ is a special class of domains, called thin domains (see \eqref{eOme}). The main theorem of this section gives the precise asymptotic behaviour of the functional $F(\Omega_{\epsilon})$
\begin{theorem}\label{tAF}
Let $\Omega_{\epsilon}\subset \mathbb{R}^2$ be a sequence of thin domains that converges to a segment in the Hausdorff sense. Then there exists a non negative and concave function $h\in L^{\infty}(0,1)$ such that the following asymptotic behaviour holds:
\begin{equation*}
F(\Omega_{\epsilon})\xrightarrow[\epsilon \rightarrow 0]{} F(h):=\frac{\mu_1(h)\int_0^1h(x)dx}{\sigma_1(h)}.
\end{equation*}
Where $\mu_1(h)$ is the first non zero eigenvalue of
\begin{equation*}
\begin{cases}
\vspace{0.3cm}
   -\frac{d}{dx}\big(h(x)\frac{d u_k}{dx}(x)\big)=\mu_k(h) h(x)u_k(x)  \qquad  x\in \big(0,1\big) \\
      h(0)\frac{du_k}{dx}(0)=h(1)\frac{du_k}{dx}(1)=0,
\end{cases}
\end{equation*}
and $\sigma_1(h)$ is the first non zero eigenvalue of
\begin{equation*}
\begin{cases}
\vspace{0.3cm}
   -\frac{d}{dx}\big(h(x)\frac{d v_k}{dx}(x)\big)=\sigma_k(h) v_k(x)  \qquad  x\in \big(0,1\big) \\
      h(0)\frac{dv_k}{dx}(0)=h(1)\frac{dv_k}{dx}(1)=0.
\end{cases}
\end{equation*}
\end{theorem}
In order to obtain this result in Lemma \ref{lAS} and in Lemma \ref{lAN} we prove general asymptotic behaviours for Neumann and Steklov eigenvalues on collapsing domains. Similar results for the Neumann eigenvalues, but in a different geometrical context, where proved in \cite{BCL21} and \cite{KT19}. We want to highlight the fact that the limit eigenvalues problems in Lemma \ref{lAS} and in Lemma \ref{lAN} are non-standard: since the function $h$ can vanish at the boundary, they are non-uniformly elliptic. We are not aware
of similar asymptotic behaviour in the literature.

In the rest of Section \ref{sCS} we are interested in studying in which way a sequence of thin domains $\Omega_{\epsilon}$ must collapse in order to obtain the lowest possible value of the limit $F(\Omega_{\epsilon})$.
From Theorem \ref{tAF} this problem is equivalent to study the minimization problem for the one-dimensional spectral functional $F(h)$ in the class of $L^{\infty}(0,1)$, concave and non negative functions. In particular in Theorem \ref{tSM} we will show that there exists a minimizer and also that the function $h\equiv 1$ is a local minimizer.

Section \ref{sULB} is devoted to the study of upper and lower bounds for the functionals $F(h)$ and $F(\Omega)$. We start by showing the following bounds for the functional $F(h)$
\begin{theorem}\label{tULBG}
For every non negative and concave function $h\in L^{\infty}(0,1)$ the following inequalities hold
\begin{equation*}
\frac{\pi^2}{12}\leq F(h)\leq 4
\end{equation*}
\end{theorem}
Then we will prove the following bounds for the functional $F(\Omega)$ 

\begin{theorem}\label{tULBF}
There exists an explicit constant $C_1$ such that, for every convex open set $\Omega\subset \mathbb{R}^2$, the following inequalities hold
\begin{equation*}
\frac{\pi^2}{6\sqrt[3]{18}}\leq F(\Omega)\leq C_1\leq 9.04
\end{equation*}
\end{theorem}
The explicit constant $C_1$ will be described in Section \ref{sULB}.

In the last Section we are interested in plotting the  Blaschke$-$Santaló diagrams
$$\mathcal{E}=\{(x,y) \mbox{ where } x=\sigma_1(\Omega)P(\Omega),\; y=\mu_1(\Omega)|\Omega|,\; \Omega\subset \R^2\}$$ 
$$\mathcal{E}^C=\{(x,y) \mbox{ where } x=\sigma_1(\Omega)P(\Omega),\; y=\mu_1(\Omega)|\Omega|,\; \Omega\subset \R^2, \;\Omega \mbox{ convex}.\}$$ 
This kind of diagrams for spectral quantities has been recently studied by different authors, let us mention for example
\cite{AH11}, \cite{BBFi}, \cite{BBP19}, \cite{HF21},  \cite{LZ19}.
In this section, we show that the diagram $\mathcal{E}$ is, in some sense, trivial while the diagram $\mathcal{E}^C$ is more complicated delimited by two unknown curves.
We present some numerical experiments and give some conjectures for this diagram.

\section{Existence or non-existence of extremal domains}\label{sE}
We show that, in general, the problem of minimization and maximization of the functional $F(\Omega)$ is ill posed, in the sense that one can construct sequences of domains for which $F(\Omega_{\epsilon})$ converge to $0$ and sequences of domains for which $F(\Omega_{\epsilon})$ converge to $+\infty$. 
\begin{proposition}\label{lIP}
The following equalities hold
\begin{equation*}
\inf\{F(\Omega):\Omega\subset \mathbb{R}^2\; \text{open and Lipschitz} \}=0,
\end{equation*}
\begin{equation*}
\sup\{F(\Omega):\Omega\subset \mathbb{R}^2\; \text{open and Lipschitz} \}=+\infty.
\end{equation*}
\end{proposition}
In order to prove that the infimum is $0$ we construct a sequence of domains $\Omega_{\epsilon}$ for which $\sigma_1(\Omega_{\epsilon})P(\Omega_{\epsilon})\rightarrow c> 0$ and $\mu_1(\Omega_{\epsilon})|\Omega_{\epsilon}|\rightarrow 0$. We use similar ideas in order to construct another sequence $\Omega_{\epsilon}$ for which $\sigma_1(\Omega_{\epsilon})P(\Omega_{\epsilon})\rightarrow 0$ and $\mu_1(\Omega_{\epsilon})|\Omega_{\epsilon}|\rightarrow c>0$, proving in this way that the supremum is $+\infty$. 

We construct the desired sequences $\Omega_{\epsilon}$ by perturbing a given set $\Omega$ by adding oscillations on the boundary (see \cite{BN20} for the details of 
the construction). Given two compact sets $\Omega_1,\Omega_2 \in \mathbb{R}^2$ we denote by $d_H(\Omega_1,\Omega_2)$ the Hausdorff distance 
between the two sets
(see \cite{HPb}), the key result is the following 
\begin{theorem}[Bucur-Nahon \cite{BN20}]\label{tBNS}
Let $\Omega,\omega\subset \mathbb{R}^2$ be two smooth, conformal open sets. Then there exists a sequence of smooth open sets $(\Omega_{\epsilon})_{\epsilon>0}$ with uniformly bounded perimeter and satisfying a uniform $\varepsilon$-cone condition (see \cite{HPb}) such that
\begin{equation}
\lim_{\epsilon\rightarrow 0}d_H(\partial\Omega_{\epsilon},\partial\Omega)=0,\;\lim_{\epsilon\rightarrow 0}P(\Omega_{\epsilon})\sigma_k(\Omega_{\epsilon})=P(\omega)\sigma_k(\omega),\; \lim_{\epsilon\rightarrow 0}|\Omega_{\epsilon}|\mu_k(\Omega_{\epsilon})=|\Omega|\mu_k(\Omega).
\end{equation}
\end{theorem}
\begin{proof}[Proof of Proposition \ref{lIP}]
Let $\delta>0$, let $\Omega$ be a simply connected domain for which $\mu_1(\Omega)|\Omega|\leq \delta$ (for example a dumbbell shape domain with the channel very thin see \cite{JM92}). Let $\omega$ be a disc, we know that $\sigma_1(\omega)P(\omega)=2\pi$. Using Theorem \ref{tBNS} we can perturb the domain $\Omega$  in such a way that 
$$
\lim_{\epsilon\rightarrow 0}P(\Omega_{\epsilon})\sigma_1(\Omega_{\epsilon})=2\pi,\ \ \lim_{\epsilon\rightarrow 0}|\Omega_{\epsilon}|\mu_1(\Omega_{\epsilon})\leq 2\delta
$$
Thus we can conclude that, for $\epsilon$ small enough
\begin{equation*}
F(\Omega_{\epsilon})\leq \frac{2\delta}{2\pi-1}
\end{equation*}
since $\delta$ was arbitrary small we conclude that:
\begin{equation*}
\inf\{F(\Omega):\Omega\subset \mathbb{R}^2\; \text{open and Lipschitz} \}=0.
\end{equation*}
For the other case, we choose $\Omega$ as the unit disc, then $\mu_1(\Omega)|\Omega|=\pi {j'}_{11}^2$ ($j'_{11}$ is the first zero of the derivative of the Bessel function
$J_1$). Let $\omega$ be a set for which $\sigma_1(\omega)P(\omega)\leq \delta$ (for example a dumbbell shape domain with the channel very thin see \cite{BHM21}), using arguments similar at the ones above we conclude that 
\begin{equation*}
\frac{\pi j_{11}^2 -1}{2\delta}\leq F(\Omega_{\epsilon}),
\end{equation*}
since $\delta$ was arbitrary small we conclude that:
\begin{equation*}
\sup\{F(\Omega):\Omega\subset \mathbb{R}^2\; \text{open and Lipschitz} \}=+\infty.
\end{equation*}
\end{proof}
We mention that there exists another way to construct a sequence of domains such that $F(\Omega_{\epsilon})\rightarrow 0$, this method is based on an 
homogenization technique, the key result is the following (see Theorem 1.14 in \cite{GKL20}):
\begin{theorem}[Girouard-Karpukhin-Lagac\'e \cite{GKL20}]
There exists a sequence of domains $\Omega_{\epsilon}\subset \mathbb{R}^2$ such that for every $k\in \mathbb{N}$ the following holds
\begin{align*}
\sigma_k(\Omega_{\epsilon})P(\Omega_{\epsilon})&\rightarrow 8\pi k\\
\mu_k(\Omega_{\epsilon})|\Omega{_\epsilon}|&\rightarrow 0
\end{align*}
\end{theorem}

From now on we will restrict ourselves to the class of convex domains. As recalled in the Introduction, a minimizing (or a maximizing) sequence of plane
convex domains $\Omega_{\epsilon}$ has the following behaviour:
\begin{enumerate}[i]
\item either the minimizing (maximizing) sequence $\Omega_{\epsilon}$ converges to a segment (for the Hausdorff metric).
\item or the minimizing (maximizing) sequence $\Omega_{\epsilon}$ converges to a convex open set $\Omega$
\end{enumerate}
In the second case (ii), we deduce that there exists a minimizer (maximizer) for the functional $F(\Omega)$ in the class of convex domains.
Indeed, the four quantities area, perimeter, $\mu_1$ and $\sigma_1$ are continuous for Hausdorff convergence of plane convex domains
(see \cite{HPb} for the first three and \cite{Bog17} or \cite{BGT20} for Steklov eigenvalues).

\section{Convex case: Thin Domains}\label{sCS}
We start by defining the following space of functions
\begin{equation}\label{eCf}
\mathcal{L}:=\{h\in L^{\infty}(0,1): h \,\text{ non negative, concave and } \int_0^1h=1\}.
\end{equation}
Given two functions $h^- \in \mathcal{L}$ and $h^+ \in \mathcal{L}$, we define the class of thin domains $\Omega_{\epsilon}$ in the following way (see Remark \ref{rCG}):
\begin{equation}\label{eOme}
\Omega_{\epsilon}=\{(x,y)\in \mathbb{R}^2 \;|\,\, 0\leq x\leq 1, \,\,-\epsilon h^-(x)\leq y\leq \epsilon h^+(x) \}.
\end{equation}

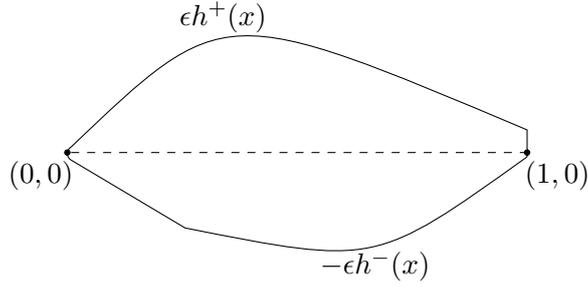
\begin{figure}[H]
\centering
\tdplotsetmaincoords{0}{0}
\begin{tikzpicture}[tdplot_main_coords, scale=1]
\draw (-3,-0.1,0)  arc (-120:-240:0.1) .. controls (-1,2,0,).. (3,0.3,0)--(3,0.05,0)arc (30:-40:0.1).. controls (1,-1.5,0,)..(-1.5,-1,0)--(-3,-0.1,0);
\draw [dashed](-3,0,0) -- (3,0,0);
\draw (-1,1.8,0) node {$\epsilon h^{+}(x)$};
\draw (1,-1.5,0) node {$-\epsilon h^{-}(x)$};
\draw (-3.4,-0.3,0) node {$(0,0)$};
\draw (3.4,-0.3,0) node {$(1,0)$};
\filldraw  (-3.05,0,0) circle (1pt);
\filldraw  (3,0,0) circle (1pt);
\end{tikzpicture}
\caption{Description of the thin domain $\Omega_{\epsilon}$}
\end{figure}
We notice that the functional $F(\Omega)$ is scale invariant so without loss of generality we can consider domains that have diameter  $D(\Omega_{\epsilon})\to 1$ when $\epsilon\to 0$.

In the next lemma we give a compactness result for the space of functions $\mathcal{L}$
\begin{lemma}\label{lCL}
Let $h_n\in \mathcal{L}$ be a sequence of functions, then there exists a function $h\in \mathcal{L}$ such that, up to a subsequence that we still denote by $h_n$, we have 
\begin{align*}
h_n &\rightarrow h \quad \text{in} \quad L^2(0,1)\\
h_n &\rightarrow h \quad \text{uniformly on every compact subset of} \; (0,1).
\end{align*}
\end{lemma}
\begin{proof}
From the concavity of the functions $h_n$ and from the fact that $||h_n||_{L^1(0,1)}=1$, we conclude that $||h_n||_{L^{\infty}(0,1)}\leq 2$. 
Let us assume first that the functions $h_n$ are smooth, say $C^1$ inside $(0,1)$. 
We fix a parameter $0<\delta<1$ and we consider the interval $I_{\delta}=[\delta, 1-\delta]$. The functions $h_n$ being uniformly bounded in $I_{\delta}$,  from the concavity and the uniform bound we conclude
\begin{equation*}
-\frac{2}{\delta}\leq -\frac{h_n(x)}{\delta}\leq h'_n(x)\leq \frac{h_n(x)}{\delta}\leq \frac{2}{\delta}\quad \forall\, x\in I_{\delta}.
\end{equation*}
We can now apply Ascoli-Arzelà Theorem and we conclude that there exists a function $h\in C([0,1])$ such that, for every $0<\delta<1$, up to a subsequence that we 
still denote by $h_n$
\begin{equation*}
h_n\rightarrow h \quad \text{uniformly in} \quad I_{\delta}.
\end{equation*}
From the convergence above and from the fact that $h_n$ is concave for every $n$ we infer that $h$ is also concave in $I_{\delta}$. So for every interval of the type $I_{\delta}$ we found the limit function $h$. 

Now we need to analyze what happens on the two extremities of the interval $[0,1]$. We consider the bounded sequence $h_n(0)$, up to a subsequence, this sequence has a limit, we extend the function $h$ that we found above to be equal at that limit in $x=0$, so $h(0)=\lim_{n\to \infty} h_n(0)$. We use the same argument for the point $x=1$. Now it is straightforward to check (by passing to the limit in the concavity inequality for $h_n$) that $h$ is a concave function on the interval $[0,1]$ and that 
\begin{equation*}
h_n\rightarrow h \quad \text{in} \quad L^2(0,1).
\end{equation*}
We finally argue by density to extend the previous result to a general sequence $h_n$.
\end{proof}

\subsection{Asymptotic behaviour of eigenvalues}
In this section we present some general results concerning the asymptotic behaviour of $\sigma_k$ and $\mu_k$ in a wide class of collapsing domains. We then apply this results in the particular case of thin domains in order to obtain the asymptotics given in Theorem \ref{tAF}.

We start with the analysis of the Steklov eigenvalues:
\begin{lemma}\label{lAS}
Let $h^+\in  L^{\infty}(0,1) $ and $h^-\in  L^{\infty}(0,1)$ be two non negative functions, we define the following collapsing domains:
\begin{equation*}
\Omega_{\epsilon}=\{(x,y)\in \mathbb{R}^2 \;|\,\, 0\leq x\leq 1, \,\,-\epsilon h^-(x)\leq y\leq \epsilon h^+(x) \}.
\end{equation*}
Let $h=h^++h^-$, if there exist $K>0$ and $p<2$ such that $h(x)\geq K(x(1-x))^p$ a. e. in $(0,1)$, then 
\begin{equation*}
\sigma_k(\Omega_{\epsilon})= \frac{\sigma_k(h)}{2}\epsilon+o(\epsilon)\quad \text{as} \quad \epsilon\rightarrow 0,
\end{equation*}
where $\sigma_k(h)$ is the $k-$th non trivial eigenvalue of
\begin{equation}\label{eES}
\begin{cases}
\vspace{0.3cm}
   -\frac{d}{dx}\big(h(x)\frac{d v}{dx}(x)\big)=\sigma(h) v(x)  \qquad  x\in \big(0,1\big) \\
      h(0)\frac{dv}{dx}(0)=h(1)\frac{dv}{dx}(1)=0.
\end{cases}
\end{equation}
\end{lemma}
\begin{remark}
In the previous Lemma the problem \eqref{eES} is understood in the weak sense. The function $h$ is allowed to vanish at the extremities of the interval, 
therefore the operator $-\frac{d}{dx}\big(h(x)\frac{d v}{dx}\big)$ is not uniformly elliptic and the existence of eigenvalues and eigenfunctions does not follow 
in a classical way. For this reason in the first part of the proof we will prove the existence of the eigenvalues, under the assumption that we made on the function $h$.
\end{remark}
\begin{proof}[Proof of Lemma \ref{lAS}]
Let $f\in L^2(0,1)$, the inverse of the operator $-\frac{d}{dx}\big(h(x)\frac{d v}{dx}\big)$ with the boundary conditions $h(0)v'(0)=h(1)v'(1)=0$ is given by the following integral representation (see \cite{T65}):
\begin{equation}\label{eIntRepS}
v(x)=\int_0^1g(x,y)f(y)dy\,\,\text{with}\,\,g(x,y)=\int^{\min(x,y)}_0\frac{t}{h(t)}dt+\int_{\max(x,y)}^1\frac{1-t}{h(t)}dt.
\end{equation}
From the  assumption on the function $h$ it follows that $g(x,y)\in L^2([0,1]\times [0,1])$. We conclude that the integral operator defined in \eqref{eIntRepS} is an Hilbert-Schmidt integral operator and so problem \eqref{eES} posses a sequence of eigenvalues and eigenfunctions. In particular the eigenvalue $\sigma_k(h)$ admits the following variational characterization:
\begin{equation}\label{eVCS1}
\sigma_k(h)=\inf_{E_k}\sup_{0\neq v\in E_k} \frac{\int_0^1(v')^2hdx_1}{\int_0^1v^2dx_1},
\end{equation}
where the infimum is taken over all $k-$dimensional subspaces of $H^1(0,1)$ which are $L^2-$ orthogonal to constants.

Let $f_k$ be the eigenfunction of the problem \eqref{eES} associated to the eigenvalue $\sigma_k(h)$, we define the function $F_k(x_1,x_2)=f_k(x_1)$ for every $(x_1,x_2)\in \Omega_{\epsilon}$. We define the mean value of the function $F_k$ on $\partial\Omega_\epsilon$:
\begin{equation*}
MF_{k,\epsilon}:=\frac{1}{P(\Omega_{\epsilon})}\int_{\partial\Omega_{\epsilon}}F_kds=\frac{1}{P(\Omega_{\epsilon})}\int_0^1f_k(\sqrt{1+(\epsilon h^{+\prime})^2}+\sqrt{1+(\epsilon h^{-\prime})^2})dx_1.
\end{equation*}
From \eqref{eES} it is straightforward to check that $\int_0^1f_k=0$, so we have the following limit
\begin{equation}\label{eLMS}
\lim_{\epsilon \to 0} MF_{k,\epsilon}=0.
\end{equation}
We introduce the following subspace $E_k=\text{Span}[F_1-MF_{1,\epsilon},...,F_k-MF_{k,\epsilon}]$, we can use this as a test subspace in the variational characterization \eqref{eVFS}, we obtain 
\begin{equation*}
\sigma_k(\Omega_{\epsilon})\leq \max_{v\in E_k} \frac{\int_{\Omega_{\epsilon}}|v|^2dx}{\int_{\partial \Omega_{\epsilon}}v^2 ds}=\max_{\beta\in \mathbb{R}^k}\frac{\epsilon \sum_{i=1}^k\beta_i^2\int_0^1 (f_i')^2hdx_1}{\int_0^1\big(\sum_{i=1}^k \beta_i(f_i-MF_{i,\epsilon})\big)^2((1+(\epsilon h^{+\prime})^2)^{\frac{1}{2}}+(1+(\epsilon h^{-\prime})^2)^{\frac{1}{2}})dx_1}.
\end{equation*}
From \eqref{eLMS} and the above inequality we can conclude that for $\epsilon$ small enough
\begin{equation}\label{eUBS}
\sigma_k(\Omega_{\epsilon})\leq\frac{\epsilon}{2} \max_{\beta\in \mathbb{R}^k}\frac{\sum_{i=1}^k\beta_i^2\int_0^1 (f_i')^2hdx_1}{\sum_{i=1}^k\beta_i^2\int_0^1 f_i^2dx_1}+o(\epsilon)=\frac{\sigma_k(h)}{2}\epsilon+o(\epsilon),
\end{equation} 
where the last equality is true because $f_k$ is the eigenfunction corresponding to $\sigma_k(h)$

On the other hand, let us denote by $\Omega_1$ the convex domain corresponding to $\epsilon=1$.
Let $v_{k,\epsilon}$ be a Steklov eigenfunction associated to  $\sigma_k(\Omega_{\epsilon})$, normalized in such a way that $||v_{k,\epsilon}||_{L^2(\partial \Omega_{\epsilon})}=1$. 
We define the following function
\begin{equation*}
\overline v_{k,\epsilon}(x_1,x_2)=v_{k,\epsilon}(x_1,\epsilon x_2)\quad \forall\; (x_1,x_2)\in \Omega_1. 
\end{equation*}
We start with the bound of $||\nabla \overline v_{k,\epsilon} ||_{L^2(\Omega_1)}$,
\begin{equation*}
\int_{\Omega_1}|\nabla \overline v_{k,\epsilon}|^2 dx\leq \int_{\Omega_1} \Big( \frac{\partial \overline v_{k,\epsilon}}{\partial x_1} \Big )^2+\frac{1}{\epsilon^2} \Big( \frac{\partial \overline v_{k,\epsilon}}{\partial x_2} \Big )^2 dx=\frac{1}{\epsilon}\int_{\Omega_{\epsilon}}|\nabla v_{k,\epsilon}|^2 dy = \frac{\sigma_k(\Omega_\epsilon)}{\epsilon}\leq C
\end{equation*}
where we did the change of coordinates $y_1=x_1$, $y_2=\epsilon x_2$ and the last inequality is true because of \eqref{eUBS}. 
We want now to bound $||\overline v_{k,\epsilon}||_{L^2(\Omega_1)}$. 
By the Poincar\'e-Friedrichs inequality or the variational characterization of Robin eigenvalues (we denote by $\lambda_1^R(\Omega,\beta)$ the first Robin eigenvalue
of the domain $\Omega$ with the boundary parameter $\beta$), we get
\begin{equation}\label{robi}
\int_{\Omega_\epsilon} v_{k,\epsilon}^2 dx \leq \frac{1}{\lambda_1^R(\Omega_\epsilon,1)}\left[\int_{\Omega_\epsilon} |\nabla v_{k,\epsilon}|^2 dx +
\int_{\partial\Omega_\epsilon} v_{k,\epsilon}^2 ds \right].
\end{equation}
Using Bossel's inequality, see \cite{Bos86}, we infer $\lambda_1^R(\Omega_\epsilon,1) \geq h(\Omega_\epsilon)- 1$ where $h(\Omega_\epsilon)$ is the Cheeger constant
of $\Omega_\epsilon$. Now by monotonicity of the Cheeger constant with respect to inclusion, we have $h(\Omega_\epsilon)\geq h(R_\epsilon)$ where $R_\epsilon$
is a rectangle of length 1 and width $4\epsilon$. Now the Cheeger constant of such a rectangle can be computed explicitly, see \cite{KLR06} and it turns out that,
for any $\epsilon$,  $h(R_\epsilon) \geq 2/\epsilon$. Therefore, using \eqref{robi} and the normalization $\int_{\partial\Omega_\epsilon} v_{k,\epsilon}^2 ds =1$ we
finally get 
$$\int_{\Omega_\epsilon} v_{k,\epsilon}^2 dx \leq \epsilon (C\epsilon +1)\leq 2\epsilon.$$
Now, coming back to $\overline v_{k,\epsilon}$, we have
$$\int_{\Omega_1} \overline v_{k,\epsilon}^2 dx=\frac{1}{\epsilon} \int_{\Omega_\epsilon} v_{k,\epsilon}^2 dx \leq 2$$
therefore we conclude that there exists $\overline{V}_k\in H^1(\Omega_1)$ such that (up to a sub-sequence that we still denote by $\overline v_{k,\epsilon}$) 
\begin{equation}\label{eCAS}
\overline v_{k,\epsilon} \rightharpoonup \overline{V}_k \quad \text{in} \quad H^1(\Omega_1), \qquad\mbox{and strongly in $L^2$}.
\end{equation}
We also know that $\overline{V}_k$ does not depend on $x_2$, indeed 
\begin{equation*}
\int_{\Omega_1} \big ( \frac{\partial \overline{v}_{k,\epsilon} }{\partial x_2} \big )^2dx=\epsilon \int_{\Omega_{\epsilon}} \big ( \frac{\partial v_{k,\epsilon} }{\partial x_2} \big )^2dx\leq C \epsilon^2 \rightarrow 0.
\end{equation*} 
We define the function $V_k$ as the restriction of $\overline{V}_k$ to the variable $x_1$. We want to prove that $\int_0^1V_kdx_1=0$ and $V_k$ is not a constant function. By definition of $\overline{v}_{k,\epsilon}$ and $v_{k,\epsilon}$
the following equality holds
\begin{equation*}
0=\int_{\partial\Omega_{\epsilon}}v_{k,\epsilon}ds=\int_0^1\overline{v}_{k,\epsilon}(x_1,h^+(x_1))\sqrt{1+(\epsilon h^{+\prime})^2}dx_1+\int_0^1\overline{v}_{k,\epsilon}(x_1,h^-(x_1))\sqrt{1+(\epsilon h^{-\prime})^2}dx_1.
\end{equation*}
Now, $\overline{v}_{k,\epsilon}$ converges strongly in $L^2$ to $\overline{V}_k$ while $\sqrt{1+(\epsilon h^{+\prime})^2}$ converges weakly in $L^2$ to 1, thus passing to the limit yields 
\begin{equation}\label{eMVS}
\int_0^1 V_kdx_1=0.
\end{equation}
Now from the fact that $||v_{k,\epsilon}||_{L^2(\partial \Omega_{\epsilon})}=1$, using similar arguments we conclude that: 
\begin{equation*}
\int_0^1 V_k^2dx_1=2,
\end{equation*}
from this equality and \eqref{eMVS} we conclude that $V_k$ cannot be a constant function. 

Using the convergence given in \eqref{eCAS}, the variational characterization and the relations that we have just obtained, we conclude that for $\epsilon$ small enough we have the following lower bound
\begin{equation}\label{eLBS}
\sigma_k(\Omega_{\epsilon})=\max_{\beta\in \mathbb{R}^k}\frac{\sum_{i=1}^k\beta_i^2\int_{\Omega_{\epsilon}}|\nabla v_{i,\epsilon}|^2dx}{\sum_{i=1}^k\beta_i^2\int_{\partial \Omega_{\epsilon}}v_{i,\epsilon}^2ds}\geq \frac{\epsilon}{2}\max_{\beta\in \mathbb{R}^k}\frac{\sum_{i=1}^k\beta_i^2 \int_0^1 V_i'^2hdx_1}{\sum_{i=1}^k\beta_i^2\int_0^1V_i^2dx_1}+o(\epsilon)\geq \frac{\sigma_k(h)}{2}\epsilon+o(\epsilon).
\end{equation}
The last inequality is true because of the variational characterization \eqref{eVCS1} for $\sigma_k(h)$. From \eqref{eUBS} and \eqref{eLBS} we finally conclude that 
\begin{equation*}
\sigma_k(\Omega_{\epsilon}) = \frac{\sigma_k(h)}{2}\epsilon+o(\epsilon)\quad \text{as} \quad \epsilon\rightarrow 0,
\end{equation*}
\end{proof}
We now specify the result above in the case of thin domains and we give also some continuity results for $\sigma_k(h)$.

\begin{lemma}\label{lCES}
Let $\Omega_{\epsilon}$ be a sequence of thin domains then Lemma \ref{lAS} holds. Moreover let $h_n\in \mathcal{L}$ and $h\in \mathcal{L}$ be such that $h_n\rightarrow h$ in $L^2(0,1)$, then,  we have
\begin{equation*}
\sigma_k(h_n)\rightarrow \sigma_k(h)
\end{equation*}
\end{lemma} 
\begin{proof}
From the concavity and positivity of $h\in \mathcal{L}$ it follows that there exists a constant $K>0$ such that 
\begin{equation}\label{eLBh}
h(x)\geq Kx(1-x)\quad \text{for a. e.}\quad 0\leq x \leq 1.
\end{equation}
In particular the hypothesis of Lemma \ref{lAS} are satisfied. 

Let $h_n\in \mathcal{L}$ and $h\in \mathcal{L}$ be such that $h_n\rightarrow h$ in $L^2(0,1)$, we define 
\begin{equation*}
v_n(x)=\int_0^1g_n(x,y)f(y)dy\,\,\text{with}\,\,g_n(x,y)=\int^{\min(x,y)}_0\frac{t}{h_n(t)}dt+\int_{\max(x,y)}^1\frac{1-t}{h_n(t)}dt.
\end{equation*}
The aim is to prove that $v_n\rightarrow v$ in $L^2(0,1)$, this, by classical results (see \cite{H06}), will imply the convergence of the spectrum. We know that up to a subsequence $h_n\rightarrow h$ a. e. in $[0,1]$, now using the lower bound \eqref{eLBh} we obtain an upper bound $g_n(x,y)\leq C$, for every $n\in \mathbb{N}$ and for every $(x,y)\in[0,1]\times [0,1]$. We can apply the dominated convergence on the sequence $g_n(x,y)$ and we conclude that $g_n(x,y)\rightarrow g(x,y)$ for every $(x,y)\in[0,1]\times [0,1]$. Similarly we can conclude also that $v_n(x)\rightarrow v(x)$ for every $x\in[0,1]$. Combining this convergence with the uniform bound on $g_n(x,y)$ we can use the dominated convergence to conclude that 
\begin{equation*}
\int_0^1(v_n(x)-v(x))^2dx\rightarrow 0.
\end{equation*}
\end{proof}

We now study the asymptotic behaviour for the Neumann eigenvalues:
\begin{lemma}\label{lAN}
Let $h^+\in  L^{\infty}(0,1) $ and $h^-\in  L^{\infty}(0,1)$ be two non negative functions, we define the following collapsing domains:
\begin{equation*}
\Omega_{\epsilon}=\{(x,y)\in \mathbb{R}^2 \;|\,\, 0\leq x\leq 1, \,\,-\epsilon h^-(x)\leq y\leq \epsilon h^+(x) \}.
\end{equation*}
Let $h=h^++h^-$, if there exist $K>0$ and $p<2$ such that $h(x)\geq K(x(1-x))^p$ a. e. in $(0,1)$, then:
\begin{equation*}
\mu_k(\Omega_{\epsilon})= \mu_k(h)+o(1)\quad \text{as} \quad \epsilon\rightarrow 0,
\end{equation*}
Where $\mu_k(h)$ is the $k-$th non trivial eigenvalue of
\begin{equation}\label{eEN}
\begin{cases}
\vspace{0.3cm}
   -\frac{d}{dx}\big(h(x)\frac{d u}{dx}(x)\big)=\mu(h) h(x)u(x)  \qquad  x\in \big(0,1\big) \\
      h(0)\frac{du}{dx}(0)=h(1)\frac{du}{dx}(1)=0,
\end{cases}
\end{equation}
\end{lemma}
\begin{proof}
Let $f\in L^2(0,1)$, the inverse of the operator $-\frac{1}{h(x)}\frac{d}{dx}\big(h(x)\frac{d u}{dx}\big)$ with the boundary conditions $h(0)u'(0)=h(1)u'(1)=0$ is given by the integral representation:
\begin{equation*}
u(x)=\int_0^1g(x,y)h(y)f(y)dy\,\,\text{with}\,\,g(x,y)=\int^{\min(x,y)}_0\frac{t}{h(t)}dt+\int_{\max(x,y)}^1\frac{1-t}{h(t)}dt.
\end{equation*}
We can adapt the proof of Lemma \ref{lAS} at this integral operator and we conclude that the problem \eqref{eEN} posses a sequence of eigenvalues and eigenfunctions. In particular the eigenvalue $\mu_k(h)$ admit the following variational characterization:
\begin{equation}\label{eVCN1}
\mu_k(h)=\inf_{E_k}\sup_{0\neq v\in E_k} \frac{\int_0^1(v')^2hdx_1}{\int_0^1v^2hdx_1},
\end{equation}
where the infimum is taken over all $k-$dimensional subspaces of $H^1(0,1)$ which are $L^2-$orthogonal to the function $h$.

Let $g_k$ be the eigenfunction associated to the eigenvalue $\mu_k(h)$, we define the function $G_k(x_1,x_2)=g_k(x_1)$ for every $(x_1,x_2)\in \Omega_{\epsilon}$. We define the mean value of the function $G_k$
\begin{equation*}
MG_{k,\epsilon}:=\frac{1}{|\Omega_{\epsilon}|}\int_{\Omega_{\epsilon}}G_kdx=\frac{1}{|\Omega_1|}\int_0^1g_khdx_1.
\end{equation*}

From \eqref{eEN} it is straightforward to check that $\int_0^1g_khdx_1=0$, so we have 
\begin{equation}\label{eLMN}
MG_{k,\epsilon}=0.
\end{equation}
We introduce the following subspace $E_k=\text{Span}[G_1,...,G_k]$, we can use this as a test subspace in the variational characterization \eqref{eVFN}, we obtain 
\begin{equation}\label{eUBN}
\mu_k(\Omega_{\epsilon})\leq \max_{\beta\in \mathbb{R}^k} \frac{\sum_{i=1}^k\beta_i^2\int_{\Omega_{\epsilon}}|\nabla G_i|^2dx}{\sum_{i=1}^k\beta_i^2\int_{\Omega_{\epsilon}}G_i^2 dx}=\max_{\beta\in \mathbb{R}^k}\frac{\sum_{i=1}^k\beta_i^2\int_0^1 (u_1')^2hdx_1}{\sum_{i=1}^k\beta_i^2\int_0^1 u_1^2hdx_1}=\mu_k(h).
\end{equation} 
where the last equality is true by the variational characterization \eqref{eVCN1} for the eigenvalue $\mu_k(h)$.

Let $u_{k,\epsilon}$ be a Neumann eigenfunction associated to $\mu_k(\Omega_{\epsilon})$, normalized in such a way that $||u_{k,\epsilon}||_{L^2(\Omega_{\epsilon})}=1$, we define the following function
\begin{equation*}
\overline{u}_{k,\epsilon}(x_1,x_2)=\epsilon^{\frac{1}{2}}u_{k,\epsilon}(x_1,\epsilon x_2)\  \forall (x_1,x_2)\in \Omega_1. 
\end{equation*}
We start with the bound of $||\nabla \overline{u}_{k,\epsilon} ||_{L^2(\Omega_1)}$,
\begin{equation*}
\int_{\Omega_1}|\nabla \overline{u}_{k,\epsilon}|^2 dx\leq \epsilon \int_{\Omega_1} \Big( \frac{\partial {u}_{k,\epsilon}}{\partial x_1} \Big )^2+\frac{1}{\epsilon^2}  \Big( \frac{\partial u_{k,\epsilon}}{\partial x_2} \Big )^2 dx \leq \int_{\Omega_{\epsilon}}|\nabla u_{k,\epsilon}|^2 dy\leq \mu_k(h)
\end{equation*}
where we did the change of coordinates $y_1=x_1$, $y_2=\epsilon x_2$, using the same change of variable we obtain $||\overline{u}_{\epsilon}||_{L^2(\Omega_1)}=1$.

We conclude that there exists $\overline{U}_k\in H^1(\Omega_1)$ such that (up to a sub-sequence that we still denote by $\overline{u}_{k,\epsilon}$) 
\begin{equation}\label{eCAN}
\overline{u}_{k,\epsilon} \rightharpoonup \overline{U}_k \quad \text{in} \quad H^1(\Omega_1) \quad \mbox{ and strongly in $L^2$}.
\end{equation}
We also know that $\overline{U}_k$ does not depend on $x_2$, indeed 
\begin{equation*}
\int_{\Omega_1} \big ( \frac{\partial \overline{U}_k}{\partial x_2} \big )^2dx\leq \liminf \int_{\Omega_1} \big ( \frac{\partial \overline{u}_{k,\epsilon} }{\partial x_2} \big )^2dx=
\liminf \epsilon^2 \int_{\Omega_{\epsilon}} \big ( \frac{\partial u_{k,\epsilon} }{\partial x_2} \big )^2dx =  0.
\end{equation*} 
We define the function $U_k$ that is the restriction of $\overline{U}_k$ to the variable $x_1$. We want to prove that $\int_0^1U_khdx_1=0$ and $U_k$ is not a constant function. By definition of $\overline{u}_{k,\epsilon}$ and $u_{k,\epsilon}$
the following equality holds
\begin{equation*}
\int_{\Omega_1}\overline{u}_{k,\epsilon}dx=\frac{1}{\epsilon^{\frac{1}{2}}}\int_{\Omega_{\epsilon}}u_{k,\epsilon}=0\quad \forall \epsilon
\end{equation*}
From the convergence results \eqref{eCAN} we know that, up to a subsequence, $\overline{u}_{k,\epsilon}$ converge a. e. to $\overline{U}_k$ so passing to the limit as $\epsilon$ goes to zero in the above equality we conclude that
\begin{equation}\label{eMVN}
\int_0^1 U_khdx_1=0.
\end{equation}
Now from the fact that $||\overline{u}_{k,\epsilon}||_{L^2(\partial \Omega_1)}=1$, using similar arguments we conclude that: 
\begin{equation*}
\int_0^1 U_k^2hdx_1=1,
\end{equation*}
from this equality \eqref{eMVN} and the fact that $\int_0^1h=1$ we conclude that $U$ cannot be a constant function. 

Using the convergence given in \eqref{eCAN} and the relations that we have just obtained, we conclude that for $\epsilon$ small enough we have the following lower bound
\begin{equation}\label{eLBN}
\mu_k(\Omega_{\epsilon})=\max_{\beta\in \mathbb{R}^k}\frac{\sum_{i=1}^k\beta_i^2\int_{\Omega_{\epsilon}}|\nabla u_{i,\epsilon}|^2dx}{\sum_{i=1}^k\beta_i^2\int_{\Omega_{\epsilon}}u_{i,\epsilon}^2ds}\geq\max_{\beta\in \mathbb{R}^k} \frac{\sum_{i=1}^k\beta_i^2\int_0^1(U_i')^2hdx_1}{\sum_{i=1}^k\beta_i^2\int_0^1U_i^2hdx_1}+o(1)\geq \mu_k(h)+o(1).
\end{equation}
The last inequality is true because because of the variational characterization \eqref{eVCN1} for $\mu_k(h)$. From \eqref{eUBN} and \eqref{eLBN} we finally conclude that 
\begin{equation*}
\mu_k(\Omega_{\epsilon})=  \mu_k(h)+o(1)\quad \text{as} \quad \epsilon\rightarrow 0,
\end{equation*}
\end{proof}

As we did for the Steklov eigenvalues, we now specify the result above in the case of thin domains and we give also some continuity results for $\mu_k(h)$.
\begin{lemma}\label{lCEN}
Let $\Omega_{\epsilon}$ be a sequence of thin domains then Lemma \ref{lAN} holds. Moreover let $h_n\in \mathcal{L}$ and $h\in \mathcal{L}$ be such that $h_n\rightarrow h$ in $L^2(0,1)$, then we have
\begin{equation*}
\mu_k(h_n)\rightarrow \mu_k(h)
\end{equation*}
\end{lemma} 
\begin{proof}
Let $f\in L^2(0,1)$, the inverse of the operator $-\frac{1}{h(x)}\frac{d}{dx}\big(h(x)\frac{d u}{dx}\big)$ with the boundary conditions $h(0)u'(0)=h(1)u'(1)=0$ is given by the integral representation:
\begin{equation*}
u(x)=\int_0^1g(x,y)h(y)f(y)dy\,\,\text{with}\,\,g(x,y)=\int^{\min(x,y)}_0\frac{t}{h(t)}dt+\int_{\max(x,y)}^1\frac{1-t}{h(t)}dt.
\end{equation*}
The proof is a straightforward adaptation of the proof of Lemma \ref{lCES} at this integral operator.  
\end{proof}

\begin{remark}\label{rCG}
We can consider the most general class of collapsing thin domains given by the following parametrization:
\begin{equation*}
\Omega_{\epsilon}=\{(x,y)\in \mathbb{R}^2 \;|\,\, 0\leq x\leq 1, \,\,-g^-(\epsilon) h^-(x)\leq y\leq g^+(\epsilon) h^+(x) \}.
\end{equation*}
Where $h^+\in  L^{\infty}(0,1) $ and $h^-\in  L^{\infty}(0,1)$ are two non negative functions that satisfy the conditions in Lemma \ref{lAS} and Lemma \ref{lAN} and $g^-(\epsilon)$, $g^+(\epsilon)$ are positive functions that go to zero when $\epsilon$ goes to zero. We define the following limit 
\begin{equation*}
\lim_{\epsilon \to 0}\frac{g^-(\epsilon)}{g^+(\epsilon)}=K<+\infty,
\end{equation*}
(if the limit above is $+\infty$ we consider the inverse and in what follows we replace $g^+(\epsilon)$ with $g^-(\epsilon)$). In this case the asymptotics of the eigenvalues $\sigma_k(\Omega_{\epsilon})$ and $\mu_k(\Omega_{\epsilon})$ become:
\begin{align*}
\sigma_k(\Omega_{\epsilon})&\sim \frac{\sigma_k(h^++Kh^-)}{2}g^+(\epsilon)+o(g^+(\epsilon))\quad \text{as} \quad \epsilon\rightarrow 0\\
\mu_k(\Omega_{\epsilon})&\sim \mu_k(h^++Kh^-)+o(1)\quad \text{as} \quad \epsilon\rightarrow 0.
\end{align*}
The proof of this asymptotics use the same arguments of the proofs of Lemma \ref{lAS} and Lemma \ref{lAN}. We prefer to give the statements and the proofs for $g^+(\epsilon)=g^-(\epsilon)=\epsilon$ in order to simplify the exposition and also because this kind of generality is not needed to study the asymptotic behaviour of $F(\Omega_{\epsilon})$.
\end{remark}
\subsection{Study of the asymptotic behaviour of $F(\Omega_{\epsilon})$}
The proof of Theorem \ref{tAF} immediately follows from the above results 
\begin{proof}[Proof of Theorem \ref{tAF}]
Without loss of generality we can rescale the sequence $\Omega_{\epsilon}$ in such a way that $D(\Omega_{\epsilon})=1$. we consider the sequence $F(\Omega_{\epsilon})$, from Lemma \ref{lAS} and Lemma \ref{lAN} we obtain the desired result by sending $\epsilon$ to zero.
\end{proof}

Let $h\in \mathcal{L}$, by Theorem \ref{tAF}, the functional
\begin{equation*}
F(h)=\frac{\mu_1(h)\int_0^1h(x)dx}{\sigma_1(h)}
\end{equation*}
describes the behaviour of the functional $F(\Omega_{\epsilon})$, when $\Omega_{\epsilon}$ is a sequence of thin domains that converges to a segment in the Hausdorff sense. We want to study the problem of finding in which way a sequence of thin domains $\Omega_{\epsilon}$ must collapse in order to obtain the lowest possible value of the limit $F(\Omega_{\epsilon})$. For this reason we prove the following theorem:
\begin{theorem}\label{tSM}
The minimization problem (resp. the maximization problem)
\begin{equation}\label{eSM}
\inf \{F(h):h\in \mathcal{L} \}, \quad (\mbox{ resp.}\,\,  \sup \{F(h):h\in \mathcal{L} \})
\end{equation}
has a solution, moreover the constant function $h\equiv 1$ is a local minimizer.
\end{theorem}
\begin{proof}
The existence of the minimizer or the maximizer follows directly from the compactness result given in Lemma \ref{lCL}, the continuity results given in Lemma \ref{lCES} and Lemma \ref{lCEN}.

The proof of the fact that $h\equiv 1$ is a local minimizer is divided in two steps where we use first and second derivative respectively.
In the first step, using the first derivative, we prove that $h\equiv 1$ satisfies a first order optimality condition and in the second step,
using second derivative, we prove that it also satisfies the second order optimality condition. 
First of all, we recall that the eigenvalues $\mu_{0,\phi}$ and $\sigma_{0,\phi}$, being the eigenvalues of a Sturm-Liouville problem, are simple eigenvalues, see e.g. \cite[chapter 5]{EgKo}.  In particular they are twice differentiable. 
Before we start the proof we fix the notation, we consider $t>0$ a positive number, and we define the following derivatives:
\begin{itemize}
\item  for every $\phi\in \mathcal{L}$ we define $\mu_{t,\phi}:=\mu_1(1+t\phi)$ and we denote by $u_{t,\phi}$ the corresponding eigenfunction. We use the following notation for the derivatives of the eigenvalues:
\begin{equation*}
\dot{\mu}_{\phi}:=\frac{d}{dt}\mu_1(1+t\phi)\Big |_{t=0}\quad \ddot{\mu}_{\phi}:=\frac{d^2}{dt^2}\mu_1(1+t\phi)\Big |_{t=0},
\end{equation*} 
and the following notation for the derivative of the eigenfunctions:
\begin{equation*}
\dot{u}_{\phi}:=\frac{d}{dt}u_{t,\phi}\Big |_{t=0}\quad \ddot{u}_{\phi}:=\frac{d^2}{dt^2}u_{t,\phi}\Big |_{t=0}.
\end{equation*} 
\item for every $\phi\in \mathcal{L}$ we define $\sigma_{t,\phi}:=\sigma_1(1+t\phi)$ and we denote by $v_{t,\phi}$ the corresponding eigenfunction. We use the following notation for the derivatives of the eigenvalues:
\begin{equation*}
\dot{\sigma}_{\phi}:=\frac{d}{dt}\sigma_1(1+t\phi)\Big |_{t=0}\quad \ddot{\sigma}_{\phi}:=\frac{d^2}{dt^2}\sigma_1(1+t\phi)\Big |_{t=0},
\end{equation*} 
and the following notation for the derivative of the eigenfunctions:
\begin{equation*}
\dot{v}_{\phi}:=\frac{d}{dt}v_{t,\phi}\Big |_{t=0}\quad \ddot{v}_{\phi}:=\frac{d^2}{dt^2}v_{t,\phi}\Big |_{t=0}.
\end{equation*} 
\end{itemize}
We notice that 
\begin{equation}\label{eU0}
\mu_{0,\phi}=\sigma_{0,\phi}=\pi^2\quad \text{and}\quad u_{0,\phi}(x)=v_{0,\phi}(x)=\sqrt{2}\cos(\pi x)
\end{equation}

\noindent{\bf Step 1.} We start by proving the following inequality
\begin{equation*}
\frac{d}{dt}F(1+t\phi)\Big |_{t=0}\geq 0\quad \forall\, \phi\in \mathcal{L}.
\end{equation*}
The derivative of $F(h)$ has the following expression
\begin{equation}\label{eGD}
\frac{d}{dt}F(1+t\phi)\Big |_{t=0}=\frac{\dot{\mu}_{\phi}}{\pi^2}+\int_0^1\phi dx-\frac{\dot{\sigma}_{\phi}}{\pi^2}.
\end{equation}
Since this kind of perturbation is classical,  see e.g. \cite[section 5.7]{HPb} we just perform a formal computation here, the complete justification would involve
an implicit function theorem together with Fredholm alternative.
We start by computing $\dot{\sigma}_{\phi}$, from \eqref{eES} we know that
\begin{equation*}
\frac{d}{dt}\Big [ -\frac{d}{dx}\big((1+t\phi)\frac{d v_{t,\phi}}{dx}\big)\Big ]\Big |_{t=0}=\frac{d}{dt}[\sigma_{t,\phi} v_{t,\phi}]\Big |_{t=0},
\end{equation*}
so we obtain the following differential equation satisfied by $\dot{v}_{\phi}$
\begin{equation}\label{eVD}
-(\phi'v_{0,\phi}'+\phi v_{0,\phi}''+\dot{v}_{\phi}'')=\dot{\sigma}_{\phi}v_{0,\phi}+\sigma_{0,\phi}\dot{v}_{\phi}.
\end{equation}
Multiplying both side of the above equation by $v_{0,\phi}$ and integrating, recalling \eqref{eU0}, we obtain 
\begin{equation}\label{eSD}
\dot{\sigma}_{\phi}=2\pi^2\int_0^1\phi\sin^2(\pi x)dx.
\end{equation} 
We now compute $\dot{\mu}_{\phi}$, from \eqref{eEN} we know that 
\begin{equation*}
\frac{d}{dt}\Big [ -\frac{d}{dx}\big((1+t\phi)\frac{d u_{t,\phi}}{dx}\big)\Big ]\Big |_{t=0}=\frac{d}{dt}[\mu_{t,\phi}(1+t\phi) u_{t,\phi}]\Big |_{t=0},
\end{equation*}
so we obtain the following differential equation satisfied by $\dot{u}_{\phi}$
\begin{equation}\label{eUD}
-(\phi'u_{0,\phi}'+\dot{u}_{\phi}'')=\dot{\mu}_{\phi}u_{0,\phi}+\mu_{0,\phi}\dot{u}_{\phi}.
\end{equation}
Multiplying both side of the above equation by $u_{0,\phi}$ and integrating, recalling \eqref{eU0}, we obtain 
\begin{equation}\label{eND}
\dot{\mu}_{\phi}=2\pi^2\int_0^1\phi(\sin^2(\pi x)-\cos^2(\pi x))dx.
\end{equation}
Using the explicit formulas given by \eqref{eSD} and \eqref{eND} in \eqref{eGD} we finally obtain 
\begin{equation*}
\frac{d}{dt}F(1+t\phi)\Big |_{t=0}=-\int_0^1\phi \cos(2\pi x)dx \quad \forall\, \phi\in \mathcal{L}.
\end{equation*}
Now it is well known (see \cite{W85}) that the first cosine Fourier coefficent of a concave function is non positive. Moreover it is easy to check that if 
$\phi\in \mathcal{L}$ then $\int_0^1\phi \cos(2\pi x)dx=0$ if and only if $\phi$ is a linear function. So we have two cases 
\begin{enumerate}[i]
\item The function $\phi\in \mathcal{L}$ is not a linear function. In this case 
\begin{equation*}
\frac{d}{dt}F(1+t\phi)\Big |_{t=0}> 0
\end{equation*}
and we conclude that $h\equiv 1$ is a local minimizer for this kind of perturbation. 
\item The function $\phi$ is of the form $\phi(x)=B+Ax$, in this case 
\begin{equation*}
\frac{d}{dt}F(1+t(B+Ax))\Big |_{t=0}=0.
\end{equation*}
\end{enumerate}
In order to conclude the proof we need to study the second variation of the functional $F(h)$ for perturbation of the form $\phi(x)=B+Ax$.

\noindent{\bf Step 2.} Given two real numbers $(A,B)\in \mathbb{R}^2\setminus (0,0)$, we want to prove that 
\begin{equation}\label{eSVG}
\frac{d^2}{dt^2}F(1+t(B+Ax))\Big |_{t=0}>0.
\end{equation}
We start by noticing that for every $k\in \mathbb{R}$ different from zero we have that $F(kh)=F(h)$, so in order to prove inequality \eqref{eSVG} it is enough to prove that
\begin{equation}
\frac{d^2}{dt^2}F(1+tAx)\Big |_{t=0}>0.
\end{equation}
This second derivative has the following expression
\begin{equation}\label{eGDD}
\frac{d^2}{dt^2}F(1+tAx)\Big |_{t=0}=\frac{\ddot{\mu}_{Ax}}{\pi^2}+\frac{\dot{\mu}_{Ax}A}{\pi^2}-\frac{\dot{\sigma}_{Ax}A}{\pi^2}-\frac{2\dot{\sigma}_{Ax}\dot{\mu}_{Ax}}{\pi^4}-\frac{\ddot{\sigma}_{Ax}}{\pi^2}+\frac{2\dot{\sigma}_{Ax}^2}{\pi^4}.
\end{equation}
From \eqref{eSD} and \eqref{eND} it is easy to check that:
\begin{equation}\label{eED}
\dot{\mu}_{Ax}=0\quad \text{and}\quad \dot{\sigma}_{Ax}=\frac{A\pi^2}{2}.
\end{equation}
We start by computing $\ddot{\sigma}_{Ax}$, from \eqref{eES} we know that
\begin{equation*}
\frac{d^2}{dt^2}\Big [ -\frac{d}{dx}\big((1+tAx)\frac{d v_{t,Ax}}{dx}\big)\Big ]\Big |_{t=0}=\frac{d^2}{dt^2}[\sigma_{t,Ax} v_{t,Ax}]\Big |_{t=0}.
\end{equation*}
After a similar computation as the one we did in order to compute $\dot{\sigma}_{\phi}$ we obtain 
\begin{equation}\label{eSDD}
\ddot{\sigma}_{Ax}=2\int_0^1Ax\dot{v}_{Ax}'v_{0,Ax}'-\dot{\sigma}_{Ax}\dot{v}_{Ax}v_{0,Ax}dx.
\end{equation}
Now we have to find the function $\dot{v}_{Ax}$ and then compute the integral above. From \eqref{eVD}, \eqref{eU0} and \eqref{eED} we can conclude that $\dot{v}_{Ax}$ satisfies the following differential equation
\begin{equation*}
-\dot{v}_{Ax}''(x)-\pi^2\dot{v}_{Ax}(x)=\big (\frac{A\pi^2}{\sqrt{2}}-Ax\sqrt{2}\pi^2 \big)\cos(\pi x)-A\sqrt{2}\pi\sin(\pi x).
\end{equation*}
We are free to choose a normalization for the eigenfunctions of the problem \eqref{eES}, so we can assume that, for every $t$, we have $\int_0^1v_{t,Ax}^2dx=1$. From this we conclude that:
\begin{equation*}
2\int_0^1\dot{v}_{Ax}v_{0,Ax}dx=\frac{d}{dt}\Big [\int_0^1v_{t,Ax}^2dx=1\Big ]\Big |_{t=0}=0.
\end{equation*}
From the boundary conditions of the problem \eqref{eES} we obtain the following boundary conditions for $\dot{v}_{Ax}$  
\begin{align*}
\dot{v}_{Ax}'(0)&=\frac{d}{dt}\Big [v_{t,Ax}'(0) \Big ]\Big |_{t=0}=0\\
\dot{v}_{Ax}'(1)&=\frac{d}{dt}\Big [(1+tA)v_{t,Ax}'(1) \Big ]\Big |_{t=0}=0.
\end{align*} 
We finally obtain that $\dot{v}_{Ax}$ must satisfy

\begin{equation*}
\begin{cases}
\vspace{0.3cm}
   -\dot{v}_{Ax}''(x)-\pi^2\dot{v}_{Ax}(x)=\big (\frac{A\pi^2}{\sqrt{2}}-Ax\sqrt{2}\pi^2 \big)\cos(\pi x)-A\sqrt{2}\pi\sin(\pi x)  \qquad  x\in \big(0,1\big) \\
\vspace{0.3cm}
\dot{v}_{Ax}'(0)=\dot{v}_{Ax}'(1)=0 \\
\int_0^1\dot{v}_{Ax}v_{0,Ax}dx=0.
\end{cases}
\end{equation*}
This problem admits a unique solution given by the following function:
\begin{equation}\label{eSVD}
\dot{v}_{Ax}(x)=\big ( \frac{A}{4\sqrt{2}}- \frac{A}{2\sqrt{2}}x\big )\cos(\pi x)+\big ( \frac{A}{2\sqrt{2}\pi}+ \frac{A\pi}{2\sqrt{2}}(x^2-x)\big )\sin(\pi x).
\end{equation}
Putting the expressions given by \eqref{eU0} and \eqref{eSVD} in the formula \eqref{eSDD} we finally obtain 
\begin{equation}\label{eSDDF}
\ddot{\sigma}_{Ax}=\frac{A^2}{8}(3-\pi^2).
\end{equation}
We now compute $\ddot{\mu}_{Ax}$, from \eqref{eEN} we know that
\begin{equation*}
\frac{d^2}{dt^2}\Big [ -\frac{d}{dx}\big((1+tAx)\frac{d u_{t,Ax}}{dx}\big)\Big ]\Big |_{t=0}=\frac{d}{dt}[\mu_{t,Ax}(1+tAx) u_{t,Ax}]\Big |_{t=0},
\end{equation*}
After a similar computation as the one we did in order to compute $\dot{\mu}_{\phi}$ we obtain 
\begin{equation}\label{eNDD}
\ddot{\mu}_{Ax}=2\int_0^1Ax(\dot{u}_{Ax}'u_{0,Ax}'-\pi^2\dot{u}_{Ax}u_{0,Ax})dx.
\end{equation}
Now we have to find the function $\dot{u}_{Ax}$ and then compute the integral above. From \eqref{eND}, \eqref{eU0} and \eqref{eED} we can conclude that $\dot{u}_{Ax}$ must satisfy the following differential equation
\begin{equation*}
-\dot{u}_{Ax}''(x)-\pi^2\dot{u}_{Ax}(x)=-A\sqrt{2}\pi\sin(\pi x).
\end{equation*}
We are free to choose a normalization for the eigenfunction of the problem \eqref{eEN}, so we can assume that for every $t$ we have $\int_0^1(1+tAx)u_{t,Ax}^2dx=1$, by differentiating with respect to t this relation and computing the derivative at zero we conclude that 
\begin{equation*}
\int_0^1\dot{u}_{Ax}u_{0,Ax}dx=A\int_0^1x\cos(\pi x)dx.
\end{equation*}
Using the same argument as above for the boundary conditions for $\dot{u}_{Ax}$ we can conclude that $\dot{u}_{Ax}$ must satisfy 
\begin{equation*}
\begin{cases}
\vspace{0.3cm}
   -\dot{u}_{Ax}''(x)-\pi^2\dot{u}_{Ax}(x)=-A\sqrt{2}\pi\sin(\pi x) \qquad  x\in \big(0,1\big) \\
\vspace{0.3cm}
\dot{u}_{Ax}'(0)=\dot{u}_{Ax}'(1)=0 \\
\int_0^1\dot{u}_{Ax}u_{0,Ax}dx=A\int_0^1x\cos(\pi x)dx.
\end{cases}
\end{equation*}
This problem admits a unique solution given by the following function:
\begin{equation}\label{eSUD}
\dot{u}_{Ax}(x)=\frac{A}{\sqrt{2}}\big (\frac{1}{\pi}\sin(\pi x)-x\cos(\pi x) \big )
\end{equation}
Putting the expressions given by \eqref{eU0} and \eqref{eSUD} in the formula \eqref{eNDD} we finally obtain 
\begin{equation}\label{eNDDF}
\ddot{\mu}_{Ax}=\frac{3}{2}A^2.
\end{equation}
Finally putting \eqref{eSDDF}, \eqref{eNDDF} and \eqref{eED} inside \eqref{eGDD} we obtain 
\begin{equation}
\frac{d^2}{dt^2}F(1+tAx)\Big |_{t=0}=\frac{A^2(9+\pi^2)}{8\pi^2}>0
\end{equation}
This concludes the proof.
\end{proof}

\section{Convex case: upper and lower bounds for $F(h)$ and $F(\Omega)$}\label{sULB}
In this section we prove Theorem \ref{tULBG} and Theorem \ref{tULBF}. For every $0<x_0<1$ we define the following triangular 
shape function 
\begin{equation*}
T_{x_0}=\begin{cases}
\vspace{0.3cm}
   \frac{x}{x_0} \qquad  x\in \big[0,x_0] \\

\frac{1-x}{1-x_0} \qquad x\in \big[x_0,1].
\end{cases}
\end{equation*} 

Before proving Theorem \ref{tULBG} let us state the following Lemma, that will be crucial in the proof of the upper bound for $F(h)$
\begin{lemma}\label{lET}
For every $0<x_0<1$ the following equality holds
\begin{equation*}
\frac{\mu_1(T_{x_0})}{\sigma_1(T_{x_0})}=4.
\end{equation*}
\end{lemma}

\begin{proof}
We want to compute the eigenvalue $\sigma_1(T_{x_0})$, we introduce the parameter $\sigma$ and we want to find a function $v\in C^1(0,1)$ such that 
\begin{equation}\label{eDEV}
\begin{cases}
\vspace{0.3cm}
   xv''(x)+v'(x)+x_0\sigma v(x)=0 \qquad  x\in \big[0,x_0] \\

(1-x)v''(x)-v'(x)+(1-x_0)\sigma v(x)=0 \qquad x\in \big[x_0,1].
\end{cases}
\end{equation}
The idea will be to solve the equation first on the interval $\big[0,x_0]$ then on the interval $\big[x_0,1]$ and then find the condition on the parameter $\sigma$ in order to 
have a good matching in the point $x_0$. Let $J_0, Y_0$ be the Bessel functions of the first and second kind respectively with parameter $0$, we start by noticing that all 
the solutions of the second order ODE \eqref{eDEV} (1st line) are given in the interval $\big[0,x_0]$ by
\begin{equation*}
v_l=C_1J_0(2\sqrt{ \sigma x_0 x}) + \hat{C}_1 Y_0(2\sqrt{ \sigma x_0 x})
\end{equation*} 
Now, since $uY'_0(u) \to 2/\pi$ when $u\to 0$ we see that, in order the boundary condition $T_{x_0}(x)v'_l(x) \to 0$ be satisfied, we must choose $\hat{C}_1=0$.
Using the change of variable $y=1-x$ is straightforward to check that, the solution of \eqref{eDEV} (2nd line) is given in the interval $\big[x_0,1]$.  by
\begin{equation*}
v_r=C_2J_0(2\sqrt{\sigma (1-x_0)(1-x)})
\end{equation*} 
Now, we impose the following matching condition $v_l(x_0)=v_r(x_0)$ and $v_l'(x_0)=v_r'(x_0)$, this condition is equivalent to say that there exists a parameter $\sigma$ for which the following system has a solution

\begin{equation*}
\begin{cases}
\vspace{0.3cm}
 C_1J_0(2\sqrt{\sigma}x_0)= C_2J_0(2\sqrt{\sigma}(1-x_0))  \\

C_1J_0'(2\sqrt{\sigma}x_0)=-C_2J_0'(2\sqrt{\sigma}(1-x_0)).
\end{cases}
\end{equation*}
The system above has a solution if and only if the parameter $\sigma$ is a root of the following transcendental equation 
\begin{equation}\label{eTES}
J_0(2\sqrt{\sigma}x_0)J_0'(2\sqrt{\sigma}(1-x_0))+J_0(2\sqrt{\sigma}(1-x_0))J_0'(2\sqrt{\sigma}x_0)=0,
\end{equation}
so $\sigma_1(T_{x_0})$ will be the smallest non zero root of the above equation.  

\medskip
Now we want to compute the eigenvalue $\mu_1(T_{x_0})$, we introduce the parameter $\mu$ and we want to find a function $u\in C^1(0,1)$ such that 
\begin{equation}\label{eDEU}
\begin{cases}
\vspace{0.3cm}
   xu''(x)+u'(x)+\mu xu(x)=0 \qquad  x\in \big[0,x_0] \\

(1-x)u''(x)-u'(x)+\mu (1-x)u(x)=0 \qquad x\in \big[x_0,1].
\end{cases}
\end{equation}
We will find the conditions on $\mu$ by using the same arguments as before. For every constant $C_1$ the following function
\begin{equation*}
u_l=C_1J_0(\sqrt{ \mu} x)
\end{equation*} 
is a solution for \eqref{eDEU} in the interval $\big[0,x_0]$ (we can rule out the function $Y_0$ by the same argument). 
Using the change of variable $y=1-x$ is straightforward to check that, for every constant $C_2$, the function  
\begin{equation*}
u_r=C_2J_0(\sqrt{\mu}(1-x))
\end{equation*} 
is a solution for \eqref{eDEU} in the interval $\big[x_0,1]$. We impose the following matching condition $u_l(x_0)=u_r(x_0)$ and $u_l'(x_0)=u_r'(x_0)$, this condition is equivalent to say that there exists a parameter $\mu$ for which the following system has a solution

\begin{equation*}
\begin{cases}
\vspace{0.3cm}
 C_1J_0(\sqrt{\mu}x_0)= C_2J_0(\sqrt{\mu}(1-x_0))  \\

C_1J_0'(\sqrt{\mu}x_0)=-C_2J_0'(\sqrt{\mu}(1-x_0)).
\end{cases}
\end{equation*}
The system above has a solution if and only if the parameter $\mu$ is a root of the following transcendental equation 
\begin{equation}\label{eTEN}
J_0(\sqrt{\mu}x_0)J_0'(\sqrt{\mu}(1-x_0))+J_0(\sqrt{\mu}(1-x_0))J_0'(\sqrt{\mu}x_0)=0,
\end{equation}
so $\mu_1(T_{x_0})$ will be the smallest non zero root of the above equation. 

Now comparing the transcendental equations \eqref{eTES} and \eqref{eTEN} we can conclude that  
\begin{equation*}
\frac{\mu_1(T_{x_0})}{\sigma_1(T_{x_0})}=4.
\end{equation*}
\end{proof}

We are now ready to prove Theorem \ref{tULBG}
\begin{proof}[Proof of Theorem \ref{tULBG}]We start by the lower bound

\noindent{\bf Lower bound.} Let $h^*=6x(1-x)$, it is known (see for instance \cite{T65}) that, for every $h\in \mathcal{L}$, the following inequality holds 
\begin{equation}\label{iUBSG}
\sigma_1(h)\leq \sigma_1(h^*)=12.
\end{equation}
Now we want to prove that, for every $h\in \mathcal{L}$, the following inequality holds  
\begin{equation}\label{iLBMG}
\mu_1(h)\geq \pi^2.
\end{equation}
Suppose by contradiction that there exists $\overline{h}\in \mathcal{L} $ such that 
\begin{equation*}
\mu_1(\overline{h})< \pi^2,
\end{equation*}
by Lemma \ref{lAN} we conclude that, for $\epsilon$ small enough, there exists a thin domain $\Omega_{\epsilon}$ such that:
\begin{equation*}
\mu_1(\Omega_{\epsilon})< \pi^2.
\end{equation*}
We reach a contradiction because we know from Payne inequality (see \cite{PW60}) that for every convex domain $\Omega$ with diameter 1
\begin{equation*}
\mu_1(\Omega)\geq \pi^2.
\end{equation*}
From \eqref{iUBSG} and \eqref{iLBMG} we conclude that, for every $h\in \mathcal{L}$, the following lower bound holds 

\begin{equation*}
\frac{\pi^2}{12}\leq F(h).
\end{equation*}

\noindent{\bf Upper bound.} We start by proving that, for every $h\in \mathcal{L}$, the following inequality holds 
\begin{equation}\label{iUBNG}
\mu_1(h)\leq \mu_1(T_{\frac{1}{2}}).
\end{equation}
Suppose by contradiction that there exists $\overline{h}\in \mathcal{L} $ such that 
\begin{equation}\label{iCUBN}
\mu_1(\overline{h})>\mu_1(T_{\frac{1}{2}}).
\end{equation} 
We introduce the following family of thin domains,  first $\Omega_\epsilon$ defined thanks to this function $\overline{h}$ and then $R_\epsilon$ defined as follows:
\begin{equation*}
R_{\epsilon}=\{(x,y)\in \mathbb{R}^2 \;|\,\, 0\leq x\leq 1, \,\,-\epsilon \frac{1}{2}T_{\frac{1}{2}}\leq y\leq \epsilon \frac{1}{2}T_{\frac{1}{2}} \},
\end{equation*}
this class of domains $R_\epsilon$ can be seen as flattering rhombi. 
By Lemma \ref{lAN} and \eqref{iCUBN} we conclude that, for $\epsilon$ small enough, we have:
\begin{equation*}
\mu_1(\Omega_{\epsilon})> \mu_1(R_{\epsilon}),
\end{equation*}
we reach a contradiction because we know from \cite{BB99}, \cite{Cheng75} that for every thin domain $\Omega_{\epsilon}$ and for every $\epsilon$ small enough
\begin{equation*}
\mu_1(\Omega_{\epsilon})\leq \lim_{\epsilon \to 0} \mu_1(R_{\epsilon})=4 j_{01}^2.
\end{equation*}
Now we prove that, for every $h\in \mathcal{L}$, the following lower bound for $\sigma_1(h)$ holds
\begin{equation}\label{iLBSG}
\sigma_1(h)\geq h(\frac{1}{2})\sigma_1(T_{\frac{1}{2}}).
\end{equation}
Let $v$ be an eigenfunction associated to $\sigma_1(h)$, using the variational characterization for $\sigma_1(h)$ and using the fact that $h$ is concave and positive we conclude that 
\begin{equation*}
\sigma_1(h)=\frac{\int_0^1(v')^2hdx}{\int_0^1v^2dx}\geq h(\frac{1}{2}) \frac{\int_0^1(v')^2T_{\frac{1}{2}}dx}{\int_0^1v^2dx}\geq h(\frac{1}{2}) \sigma_1(T_{\frac{1}{2}}),
\end{equation*} 
where in the last inequality we used the variational characterization for $\sigma_1(T_{\frac{1}{2}})$. From \eqref{iUBNG} and \eqref{iLBSG} we conclude that:
\begin{equation}
F(h)\leq \frac{\mu_1(T_{\frac{1}{2}})}{\sigma_1(T_{\frac{1}{2}})}\frac{\int_0^1hdx}{h(\frac{1}{2})}\leq 4,
\end{equation}
where the last inequality comes from the fact that $h\in \mathcal{L}$ and Lemma \ref{lET}.

\end{proof}

We turn to the proof of Theorem \ref{tULBF}.
Let $\tau\in [0,1]$ be a parameter, in order to prove the upper bound in Theorem \ref{tULBF}, we need to introduce the following family of 
polynomials of degree four:

\begin{equation*}
P_{\tau}(y)=\frac{1}{4}\tau y^4-2y^3+5\tau y^2-4\tau^2 y+\tau^3.
\end{equation*}
In the next Lemma we prove that the polynomials $P_{\tau}$ have always positive roots and we give some explicit estimates on its roots, this estimates will be useful
in the proof of the upper bound for $F(\Omega)$.
\begin{lemma}\label{lPOL}
Let $0<\tau<1$, then the polynomial $P_{\tau}$ has four positive roots. Let $\{y_1(\tau),y_2(\tau),y_3(\tau),y_4(\tau)\}$ be its roots ordered in increasing order, then the following holds: 
\begin{enumerate}[i]
\item if $0<\tau\leq \frac{\sqrt{3}}{2}$, then $y_1(\tau)\in (0,\frac{2}{3}\tau)$, $y_2(\tau)\in (\frac{2}{3}\tau,\tau+\frac{1}{2}\tau^2)$, $y_3(\tau)\in (\tau+\frac{1}{2}\tau^2,2+\sqrt{2})$ and $y_4(\tau)\in (2+\sqrt{2},+\infty)$

\item if $\frac{\sqrt{3}}{2}\leq \tau\leq 0.9 $, then $y_1(\tau)\in (0,\frac{1}{2})$, $ y_2(\tau)\in (\frac{1}{2},\tau+\frac{1}{2}\tau^2)$, $y_3(\tau)\in (\tau+\frac{1}{2}\tau^2,2+\sqrt{2})$ and $y_4(\tau)\in (2+\sqrt{2},+\infty)$,

\item if $0.9\leq \tau < 1 $, then $y_1(\tau)\in (0,2-\sqrt{2})$, $y_2(\tau)\in (2-\sqrt{2},\tau+\frac{1}{2}\tau^2)$, $y_3(\tau)\in (\tau+\frac{1}{2}\tau^2,2+\sqrt{2})$ and  $y_4(\tau)\in (2+\sqrt{2},+\infty)$.
\end{enumerate}
Moreover $P_{\tau}(y)\geq 0$ in $[0,y_1(\tau)]\cup[y_2(\tau),y_3(\tau)]\cup [y_4(\tau),+\infty)$.
\end{lemma}
\begin{proof}
We start by noticing that for every $0<\tau<1$ we have that $P_{\tau}(0)>0$ and $\lim_{y\to +\infty}P_{\tau}(y)=+\infty$, the idea of the proof will be to find three consecutive points $0<a<b<c<+\infty$ for wich $P_{\tau}(a)<0$, $P_{\tau}(b)>0$ and $P_{\tau}(c)<0$. Before passing to the three different cases, we give some inequalities that are true for every $0<\tau<1$. It is straightforward to check that the following inequalities hold:
\begin{equation}\label{ePT1}
P_{\tau}\big (\tau+\frac{1}{2}\tau^2 \big )=\frac{1}{4}\tau^6\big (1+\frac{3}{2}\tau+\frac{1}{2}\tau^2+\frac{1}{4}\tau^3 \big)>0 \;\;\; \forall \; 0<\tau<1,
\end{equation}
\begin{equation}\label{ePT2}
P_{\tau} (2+\sqrt{2} )=(\tau-1)\big (\tau^2-(7+4\sqrt{2})\tau +40+28\sqrt{2})<0 \;\;\; \forall \; 0<\tau<1.
\end{equation}
We now prove separately the three cases.
\begin{enumerate}[i]
\item If $0<\tau\leq \frac{\sqrt{3}}{2}$, then the following inequality holds 
\begin{equation*}
P_{\tau} (\frac{2}{3}\tau)=\frac{4}{9}\tau^3 \big(\frac{\tau^2}{9}-\frac{1}{12} \big )<0,
\end{equation*}
the result follows from this inequality combined with \eqref{ePT1} and \eqref{ePT2}.
 
\item if $\frac{\sqrt{3}}{2}\leq \tau\leq 0.9 $, then the following inequalities hold 
\begin{align*}
P_{\tau} (\frac{1}{2})&=\tau^3-2\tau^2+\frac{81}{64}\tau-\frac{1}{4}<0,\\
\frac{1}{2}&<\tau+\frac{1}{2}\tau^2,
\end{align*}
the result follows from the inequalities above combined with \eqref{ePT1} and \eqref{ePT2}.

\item If $0.9\leq \tau < 1 $, then the following inequalities hold
\begin{align*}
P_{\tau} (2-\sqrt{2})&=(\tau-1)\big (\tau^2-(7-4\sqrt{2})\tau +40-28\sqrt{2})<0,\\
2-\sqrt{2}&<\tau+\frac{1}{2}\tau^2.
\end{align*}
the result follows from the inequalities above combined with \eqref{ePT1} and \eqref{ePT2}.
\end{enumerate}
\end{proof}
We now state Theorem \ref{tULBF} in a more precise way, in order to give more information about the explicit constant $C_1$.
\begin{theorem}
Let $K$ be the following constant 

\begin{equation*}
K=\max_{\tau\in [0,1]}\frac{2\pi \tau}{y_2(\tau)[2\sqrt{1-\tau^2}+2\tau \arcsin(\tau)]}.
\end{equation*}

Then for every bounded convex open set $\Omega\subset \mathbb{R}^2$, the following inequalities hold
\begin{equation*}
\frac{\pi^2}{6\sqrt[3]{18}}\leq F(\Omega)\leq2(1+K)\leq 9.04.
\end{equation*}
\end{theorem}

\begin{proof} We start by proving the lower bound

\noindent{\bf Lower bound.}
Let $\delta\in [2,\pi] $, we define the following class of bounded convex domains 
\begin{equation}
\mathcal{C}_{\delta}:=\{\Omega\subset \mathbb{R}^2: \Omega \text{ is convex and }P(\Omega)\leq \delta D(\Omega)\}.
\end{equation}
We recall that the functional $F(\Omega)$ is invariant under translation and rotation, so without loss of generality, we can assume that the origin is the center of mass of the boundary of $\Omega$ and the $x_1$ axis is parallel to (one of) the diameter(s). We know the following inequalities for $\mu_1(\Omega)$ and  $\sigma_1(\Omega)$
\begin{equation*}
\mu_1(\Omega)\geq \frac{\pi^2}{D(\Omega)^2}, \quad \sigma_1(\Omega)\leq \frac{|\Omega|}{\int_{\partial \Omega}x_1^2ds}\leq \frac{6|\Omega|}{D(\Omega)^3}.
\end{equation*}
The inequality for $\mu_1$ is Payne inequality (see \cite{PW60}) and the inequality for $\sigma_1(\Omega)$ is obtained by using the function $u(x_1,x_2)=x_1$ as a test function in \eqref{eVFN} and then using the fact that $\int_{\partial \Omega}x_1^2ds\geq \int_{-\frac{D}{2}}^{\frac{D}{2}}x_1^2dx_1$. Let $\Omega\in \mathcal{C}_{\delta}$, using the inequalities above we obtain 
\begin{equation}\label{eLB1}
F(\Omega)\geq \frac{\pi^2}{6\delta}.
\end{equation}
Now we consider the class of domains $\mathcal{C}_{\delta}^{\mathrm{c}}$, i. e. convex domains such that $P(\Omega)> \delta D(\Omega)$. We start by recalling the following result (see \cite{S60} for a geometric proof or \cite{H85} for a proof based on Fourier series)
\begin{equation}\label{eMI}
\min \Big \{ \frac{\int_{\partial \Omega}(x_1^2+x_2^2)ds}{P(\Omega)^3} : \Omega\subset \mathbb{R}^2 \text{ convex}\Big \}=\frac{1}{54},
\end{equation}  
and the minimum is achieved by the equilateral triangle.  Assuming that the origin is at the  center of mass of the boundary, and using in the variational characterization 
\eqref{eVFS} the coordinates functions $x_1$ and $x_2$ we obtain after summing
\begin{equation}\label{eUBSI}
\sigma_1(\Omega)\leq \frac{2|\Omega|}{\int_{\partial \Omega}(x_1^2+x_2^2)ds}.
\end{equation} 
Now from Payne inequality,  ($\mu_1(\Omega)\geq \pi^2/D^2$), \eqref{eMI} and \eqref{eUBSI} we conclude that for every $\Omega\in \mathcal{C}_{\delta}^{\mathrm{c}}$ the following holds
\begin{equation}\label{eLB2}
F(\Omega)\geq \frac{\delta^2 \pi^2}{108}.
\end{equation}
We notice that the lower bounds in \eqref{eLB1} and \eqref{eLB2} coincide when $\delta=\sqrt[3]{18}$, so we finally obtain:
\begin{equation*}
F(\Omega)\geq \frac{\pi^2}{6\sqrt[3]{18}}
\end{equation*}

\noindent{\bf Upper bound.}
Given a bounded convex set $\Omega\subset \mathbb{R}^2$, we denote by $r(\Omega)$ its inradius and by $w(\Omega)$ its minimal width. We know the 
following estimate from below for $\sigma_1(\Omega)$ (see \cite{KS68})
\begin{equation*}
\sigma_1(\Omega)\geq \frac{\mu_1(\Omega)r(\Omega)}{2(1+\sqrt{\mu_1(\Omega)}D(\Omega))},
\end{equation*}

we also know the following upper bound for $\mu_1(\Omega)$, see \cite{HLL21}:
\begin{equation*}
\mu_1(\Omega)\leq \pi^2 \frac{w(\Omega)^2}{|\Omega|^2},
\end{equation*}
we also use the following geometric inequality (see \cite{B29})
\begin{equation*}
\frac{|\Omega|}{r(\Omega) P(\Omega)}\leq 1.
\end{equation*}
Using the three inequalities above we conclude that 
\begin{equation}\label{eUBF}
F(\Omega)\leq 2\Big (1+\frac{\pi w(\Omega) D(\Omega)}{r(\Omega) P(\Omega)}\Big ).
\end{equation}
We introduce the parameter $\tau= \frac{w(\Omega)}{D(\Omega)}$, we know the following geometric inequality (see \cite{Ku1923}, \cite{SA00})
\begin{equation*}
\frac{D(\Omega)}{P(\Omega)}\leq \frac{1}{2\sqrt{1-\tau^2}+2\tau \arcsin(\tau)}=:g(\tau).
\end{equation*}
Now, in order to obtain an upper bound for the functional $F(\Omega)$, we need an upper bound for the quantity $\frac{w(\Omega)}{r(\Omega)}$ where the quantity $\tau= \frac{w(\Omega)}{D(\Omega)}$ is fixed.

The complete system of inequalities for the triplet $(w(\Omega), D(\Omega), r(\Omega))$ is known, in \cite{H00}we can find the Blaschke$-$Santaló diagram where $x(\Omega)=\tau=\frac{w(\Omega)}{D(\Omega)}$ and $y(\Omega)=\frac{2r(\Omega)}{D(\Omega)}$. Let us fix the the quantity $\tau$, in order to obtain an upper bound for $\frac{w(\Omega)}{r(\Omega)}$ it is enough to obtain a lower bound for $y(\Omega)$. From \cite{H00} we know that the following inequality holds
\begin{equation*}
P_{\tau}(y(\Omega))=\frac{1}{4}\tau y(\Omega)^4-2y(\Omega)^3+5\tau y(\Omega)^2-4\tau^2 y(\Omega)+\tau^3\geq 0.
\end{equation*}
In particular from Lemma \ref{lPOL} we know that $y(\Omega)\in [0,y_1(\tau)]\cup[y_2(\tau),y_3(\tau)]\cup [y_4(\tau),+\infty)$, we now prove that $y(\Omega)\geq  y_2(\tau)$. Suppose by contradiction that $y(\Omega)\in [0,y_1(\tau)]$, from the Blaschke$-$Santaló diagram $(w(\Omega), D(\Omega), r(\Omega))$ we see that $y(\Omega)\geq \frac{2}{3} \tau$, but now from Lemma \ref{lPOL} we know that $y_1(\tau)<\frac{2}{3} \tau$ and this is a contradiction. Note that we can prove in the same way that $y(\Omega) < y_4(\tau)$.

We conclude that $y(\Omega)\geq y_2(\tau)$, so we finally obtain the following upper bound 
\begin{equation*}
\frac{\pi w(\Omega) D(\Omega)}{r(\Omega) P(\Omega)}\leq \frac{2\pi g(\tau)\tau}{y_2(\tau)}=:f(\tau).
\end{equation*}
We introduce the following constant 
\begin{equation*}
K=\max_{\tau\in [0,1]} f(\tau)
\end{equation*}
numerically one can check that $K\leq 3.52$ (see Figure \ref{fFT}), from \eqref{eUBF} we finally conclude that 
 
\begin{equation*}
\frac{\pi^2}{6\sqrt[3]{18}}\leq F(\Omega)\leq2(1+K)\leq 9.04.
\end{equation*}

\begin{center}
\begin{figure}
\includegraphics[width=0.5\textwidth]{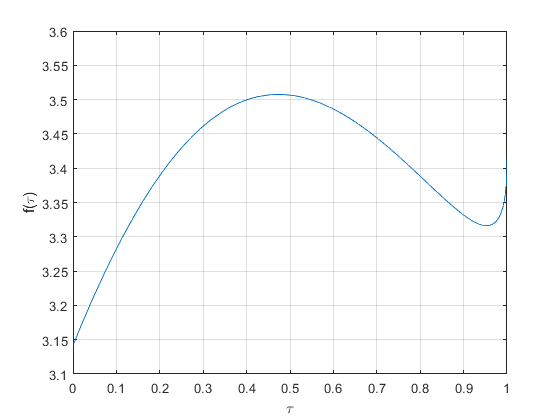}
\caption{Plot of the function $f(\tau)$}\label{fFT}
\end{figure}
\end{center}
\end{proof}

\begin{center}
\begin{figure}
\includegraphics[width=0.5\textwidth]{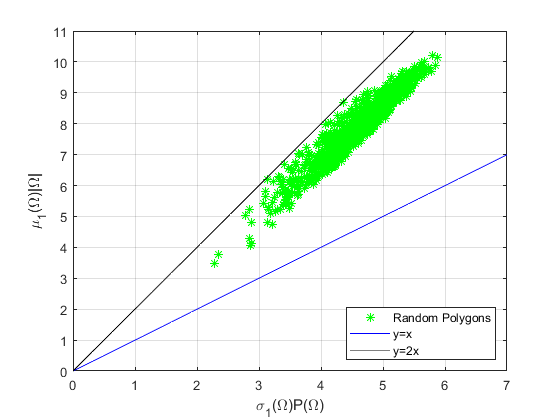}
\caption{Blaschke$-$Santaló diagram with random convex polygons}\label{fRP}
\end{figure}
\end{center}

\section{Blaschke$-$Santaló diagrams and open problems}\label{sBS}
A Blaschke-Santaló diagram is a convenient way to represent in the plane the possible values taken by two quantities (geometric or spectral).
As mentioned in the Introduction, such a diagram has been recently established for quantities like $(\lambda_1(\Omega),\lambda_2(\Omega))$ (the Dirichlet
eigenvalues) in \cite{AH11}, \cite{BBFi}, $(\mu_1(\Omega)\,\mu_2(\Omega))$ (the Neumann eigenvalues) in \cite{AH11},  $(\lambda_1(\Omega),\mu_1(\Omega))$ in
\cite{HF21} or $(\lambda_1(\Omega),T(\Omega))$ (where $T(\Omega)$ is the torsion) in \cite{BBP19}, \cite{LZ19}.

Here we are interested in plotting the set of points $(x,y)$ with
$$\mathcal{E}=\{(x,y) \mbox{ where } x=\sigma_1(\Omega)P(\Omega),\; y=\mu_1(\Omega)|\Omega|,\; \Omega\subset \R^2\}$$ 
$$\mathcal{E}^C=\{(x,y) \mbox{ where } x=\sigma_1(\Omega)P(\Omega),\; y=\mu_1(\Omega)|\Omega|,\; \Omega\subset \R^2, \;\Omega \mbox{ convex}.\}.$$ 

\subsection{The Blaschke-Santaló diagram $\mathcal{E}$}
We start with the diagram $\mathcal{E}$ (no constraint on the sets $\Omega$).
\begin{theorem}
The following equality holds 
\begin{equation*}
\overline{\mathcal{E}}=[0,8\pi]\times[0, \mu_1(\mathbb{D})\pi ]
\end{equation*} 
where $\mu_1(\mathbb{D})={j'}_{11}^2$ is the first Neumann eigenvalue of the unit disk.
\end{theorem}
\begin{proof}
 We recall the following classical result by Szeg\"o (for the simply connected case) and Weinberger \cite{Szeg54} and \cite{W56}.
\begin{equation*}
\max\{\mu_1(\Omega)|\Omega| \;|\,\, \Omega\subset \mathbb{R}^2\; \text{bounded, open and Lipschitz}\}=\mu_1(\mathbb{D})\pi,
\end{equation*}
from \cite{GKL20} we also know that 
\begin{equation*}
\sup\{\sigma_1(\Omega)P(\Omega) \;|\,\, \Omega\subset \mathbb{R}^2\; \text{bounded, open and Lipschitz}\}=8\pi.
\end{equation*}
From the inequalities above it is clear that $\mathcal{E} \subset  [0,8\pi]\times[0, \mu_1(\mathbb{D})\pi ] $, now we want to prove that $[0,8\pi)\times[0, \mu_1(\mathbb{D})\pi ]\subseteq \overline{\mathcal{E}}$.

We start by proving that for every $y\in [0, \mu_1(\mathbb{D})\pi ]$ there exists a simply connected domain $\Omega_y$ for which $\mu_1(\Omega_y)|\Omega_y|=y$. 
For that purpose, let us consider a dumbbell domain $D_{\epsilon}$, we know that we can choose the width of the channel in order to have 
$\mu_1(D_{\epsilon})|D_{\epsilon}|=\epsilon$ where $\epsilon$ is a small quantity, (see \cite{JM92}).  Now we can gradually enlarge the channel (preserving the 
$\epsilon$-cone condition) until we reach a stadium, then we can modify this stadium continuously until we reach the ball.  In all that process, the eigenvalue $\mu_1$ and 
the area vary continuously. So we constructed a continuous path for the value $\mu_1(\Omega_y)|\Omega_y|$ starting from $\epsilon$ and arriving to $\mu_1(\mathbb{D})\pi$, 
we conclude because $\epsilon$ was arbitrary small. Using the same argument (and \cite{BHM21}) we can prove that for every $x\in [0, 2\pi ]$ there exists a simply 
connected domain $\Omega_x$ for which $\sigma_1(\Omega_x)P(\Omega_x)=x$ ($2\pi$ is the value of $P(\mathbb{D}) \sigma_1(\mathbb{D)}.$).

Let $(x,y)\in [0,8\pi]\times[0, \mu_1(\mathbb{D})\pi]$ we want to prove that there exists a sequence of domains $\Omega_{\epsilon}$ such that 
$\sigma_1(\Omega_{\epsilon})P(\Omega_{\epsilon})\rightarrow x$ and $\mu_1(\Omega_{\epsilon})|\Omega_{\epsilon}|\rightarrow y$. From the discussion above we know 
that there exists a simply connected domain $\Omega_y$ for which $\mu_1(\Omega_y)|\Omega_y|=y$, now we divide the proof in two cases: 

\noindent{\bf Case 1.} Suppose $x> \sigma_1(\Omega_y)P(\Omega_y)$, let $\beta$ be a non negative and non trivial function, we introduce the following weighted Neumann eigenvalue 
\begin{equation*}
\mu_1(\Omega,\beta)=\min \Big \{  \frac{\int_{\Omega}|\nabla u|^2dx}{\int_{\Omega}u^2\beta  dx}:u\in H^1(\Omega),\int_{\Omega}u\beta dx=0 \Big \}.
\end{equation*}
From Theorem $1.11$ in \cite{GKL20} we know that for every domain $\Omega$ and every non negative and non trivial function $\beta \in L^1(\log L)^1$ (this space is a Orlicz space see \cite{GKL20} for the details) there exists a sequence of subdomains $\Omega_{\epsilon}\subseteq \Omega$ such that
\begin{align*}
\sigma_1(\Omega_{\epsilon})P(\Omega_{\epsilon})&\rightarrow \mu_1(\Omega,\beta)\int_{\Omega}\beta dx,\\
\mu_1(\Omega_{\epsilon})|\Omega_{\epsilon}|&\rightarrow \mu_1(\Omega)|\Omega|.
\end{align*}
Let us fix a parameter $\delta$, from \cite{GKL20} we know that there exists a function $\beta_1$ such that $\mu_1(\Omega,\beta_1)\int_{\Omega}\beta_1 dx\leq 8\pi-\delta$, we also know (see \cite{LP15}) that there exists a function $\beta_2$ such that $|\mu_1(\Omega,\beta_2)\int_{\Omega}\beta_2 dx-\sigma_1(\Omega)P(\Omega)|\leq \delta$. Let $0\leq t\leq 1$, we consider the following family of functions $\beta_t=t\beta_1+(1-t)\beta_2$ and we introduce the measures $d\mu_t=\beta_tdx$. It is straightforward to check that the family of measures $d\mu_t$ satisfies the conditions \textbf{M$1$}, \textbf{M$2$} and \textbf{M$3$} in page $26$ of \cite{GKL20}, in particular for every $z\in [\sigma_1(\Omega)P(\Omega)+\delta,8\pi-\delta]$ there exists $t\in[0,1]$ such that $\mu_1(\Omega,\beta_{t})\int_{\Omega}\beta_{t} dx=z$. 

We know that $x\in [\sigma_1(\Omega_y)P(\Omega_y)+\delta,8\pi-\delta]$, let $t_0$ be such that $\mu_1(\Omega_y,\beta_{t_0})\int_{\Omega_y}\beta_{t_0} dx=x$, from the previous results we conclude that there exists a sequence of domains $\Omega_{\epsilon}\subseteq \Omega_y $ such that 
$$\begin{array}{c}
\sigma_1(\Omega_{\epsilon})P(\Omega_{\epsilon}) \rightarrow \mu_1(\Omega_y,\beta_{t_0})\int_{\Omega}\beta_{t_0} dx=x\\
\mu_1(\Omega_{\epsilon})|\Omega_{\epsilon}| \rightarrow \mu_1(\Omega_y)|\Omega_y|=y.
\end{array}$$
The result follows because $\delta$ was arbitrary.

\noindent{\bf Case 2.} Suppose $x\leq \sigma_1(\Omega_y)P(\Omega_y)$, form the fact that $\Omega_y$ is simply connected we know from \cite{Wei54}
that $x\leq 2\pi$. By a previous step we know that there exists a simply connected domain $\omega$ such that $\sigma_1(\omega)P(\omega)=x$, now from Theorem 
\ref{tBNS} (see \cite{BN20} for details) we know that there exists a sequence of smooth open sets $\Omega_{\epsilon}$ such that 
$$\begin{array}{c}
\sigma_1(\Omega_{\epsilon})P(\Omega_{\epsilon}) \rightarrow \sigma_1(\omega)P(\omega)=x\\
\mu_1(\Omega_{\epsilon})|\Omega_{\epsilon}| \rightarrow \mu_1(\Omega_y)|\Omega_y|=y.
\end{array}$$
This concludes the proof.
\end{proof}
We can give the following more precise conjecture:
\begin{conjecture}
Prove that $\mathcal{E}=(0,8\pi)\times (0,\pi \mu_1(\mathbb{D}))\cup \{(0,0)\} \cup \{(2\pi, \pi \mu_1(\mathbb{D}))\}.$
\end{conjecture}
The point $\{(0,0)\}$ is attained by any disconnected domain. Moreover the segments $\{0\}\times(0,\pi \mu_1(\mathbb{D}))$ and
$(0,8\pi)\times \{0\}$ cannot be in the set $\mathcal{E}$ because if $\mu_1$ or $\sigma_1$ are zero, it means that the domain is disconnected,
thus $(\sigma_1,\mu_1)=(0,0)$. The segment $(0,8\pi)\times  \{\pi \mu_1(\mathbb{D})\}$ only contains the point corresponding to the disk
because the disk is the only domain providing equality in the Szeg\"o-Weinberger inequality. Finally, the segment 
$\{8\pi\}\times (0,\pi \mu_1(\mathbb{D}))$ is not included in the diagram because the inequality $P(\Omega)\sigma_1(\Omega)<8\pi$
is strict, see \cite{GKL20}.Thus the conjecture means that except these "boundary lines", every point $(x,y)$ such that
$0<x<8\pi$ and $0<y< \pi \mu_1(\mathbb{D})$ should correspond to a set $\Omega$ in the sense that $x=P(\Omega) \sigma_1(\Omega)$
and $y=|\Omega| \mu_1(\Omega)$.
\subsection{The Blaschke-Santaló diagram $\mathcal{E^C}$}
Now we turn to the convex case. To have some idea about the shape of this diagram, we produced random convex polygons in the plane 
and plot the corresponding
quantities $x=\sigma_1(\Omega)P(\Omega),\; y=\mu_1(\Omega)|\Omega|$.

Figure \ref{fRP} shows the values of these quantities for $1000$ random convex polygons. Each of this polygon is constructed by choosing $15$ random points in the plane and then we compute the convex hull of this points. From Figure \ref{fRP} it is natural to conjecture that $1\leq F(\Omega)\leq 2$.

Now we show some experiments that will give us informations about the behaviour of the extremal sets in the class of convex domains. In the  Figure \ref{fRT} we plotted the quantities $\sigma_1(\Omega)P(\Omega)$ and $\mu_1(\Omega)|\Omega|$ for random triangles in the plane.  
\begin{center}
\begin{figure}[H]
\includegraphics[width=0.5\textwidth]{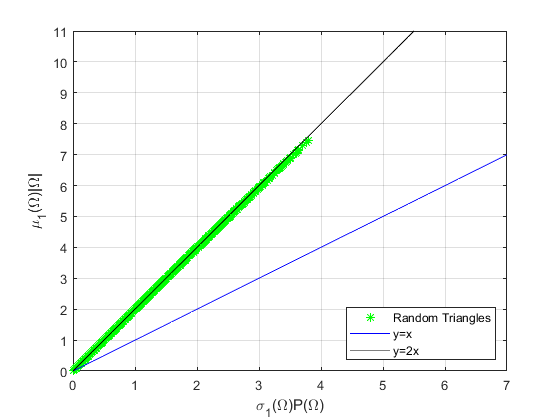}
\caption{Blaschke$-$Santaló diagram with random triangles}\label{fRT}
\end{figure}
\end{center}

From Figure \ref{fRT} we see that for every triangle $T\subset \mathbb{R}^2$ we have that $F(T)$ is slightly less than (and very close to) $2$. Actually a more precise numerical computation shows that it is not true that $F(T)=2$ for every triangles. For example, let $T_1$ be an equilateral triangle of length $1$, we know that $\mu_1(T_1)=\frac{16 \pi^2}{9}$ and let $T_2$ be a right triangle with both cathetus equal to $1$, we know that $\mu_1(T_2)=\pi^2$. A precise numerical computation of the first Steklov eigenvalue for $T_1$ and $T_2$ (using $P2$ finite element methods) gives us the following values $\sigma_1(T_1)\approx 1.2908$ and $\sigma_1(T_2)\approx 0.7310$. Using these values inside the functional $F(\Omega)$ we finally obtain
\begin{align*}
F(T_1)\approx 1.962<2,\\
F(T_2)\approx 1.977<2.
\end{align*}

The value $2$ can be reached asymptotically, let us  consider the following sequence of collapsing  triangles
\begin{equation*}
\Omega_\epsilon=\{(x,y)\in \mathbb{R}^2 \;|\,\, 0\leq x\leq 1, \,\, 0\leq y\leq \epsilon T_{\frac{1}{2}} \},
\end{equation*}
from Theorem \ref{tAF} and Lemma \ref{lET} we conclude that 
\begin{equation*}
F(\Omega_\epsilon)\rightarrow F(T_{\frac{1}{2}})=2.
\end{equation*}
We remark that, from Theorem \ref{tAF} and Lemma \ref{lET}, $F(\Omega_{\epsilon})\rightarrow 2$ for every sequence $\Omega_{\epsilon}$ of collapsing thin domains for which $h=h^++h^-=T_{x_0}$, where $0<x_0<1$.

It remains to characterize the behaviour of the minimizing sequence. We introduce the following family of collapsing rectangles:
\begin{equation*}
C_{\epsilon}=\{ (x,y)\in \mathbb{R}^2 \;|\,\, 0\leq x\leq 1, \,\, 0\leq y\leq \epsilon \}.
\end{equation*}
We plot the values of $\sigma_1(C_{\epsilon})P(C_{\epsilon})$ and $\mu_1(C_{\epsilon})|C_{\epsilon}|$ when $\epsilon$ is approaching zero.

\begin{center}
\begin{figure}[H]
\includegraphics[width=0.5\textwidth]{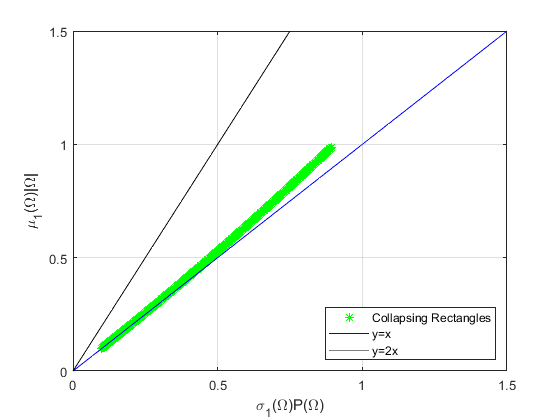}
\caption{Blaschke$-$Santaló diagram with collapsing rectangles}\label{fCT}
\end{figure}
\end{center}

We know from Theorem \ref{tAF} that $F(C_{\epsilon})\rightarrow 1$ but from Figures \ref{fCT} and \ref{fRP} it seems that $F(\Omega)>1$ for every 
$\Omega\subset \mathbb{R}^2$ convex and the only way to approach the value 1 is given by a sequence of collapsing rectangles. 

Supported by these numerical evidences we state the following conjectures:

\begin{conjecture}
For every bounded, convex and open set $\Omega\subset \mathbb{R}^2$ the following bounds hold 
$$
1\leq F(\Omega)\leq 2.
$$
\end{conjecture}

\begin{conjecture}
The following minimization problem has no solution
$$
\inf\{F(\Omega)\;|\,\, \Omega\subset \mathbb{R}^2 \,\,\text{bounded, convex and open} \}.
$$
In particular every minimizing sequence $\Omega_{\epsilon}$ must be of the form of collapsing rectangles.
\end{conjecture}

We now consider only convex quadrilaterals in $\mathbb{R}^2$, in the following numerical experiment we will have in red random convex quadrilaterals and in green collapsing rectangles, starting form a square $\mathbb{S}$ of unit area (corresponding to the farthest green point from the origin) and asymptotically approach the segment. 
\begin{center}
\begin{figure}[H]
\includegraphics[width=0.5\textwidth]{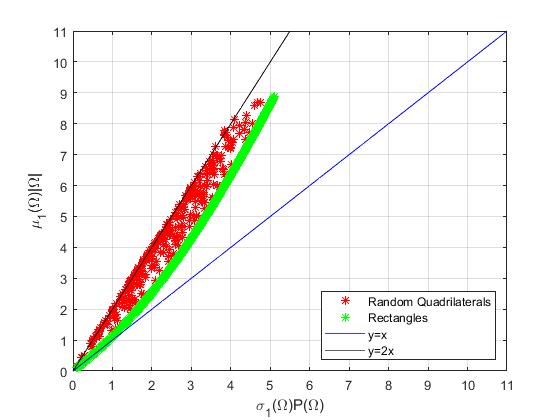}
\caption{Blaschke$-$Santaló diagram with random convex quadrilaterals and collapsing rectangles}\label{fQR}
\end{figure}
\end{center}
From Figure \ref{fQR} it is natural to state the following conjecture.
\begin{conjecture}
For every $0<C\leq 4\sigma_1(\mathbb{S})$ the solution of the minimization problem 
\begin{equation*}
\inf \{\mu_1(\Omega)|\Omega|\;|\,\, \Omega\subset \mathbb{R}^2 \,\, \text{convex quadrilateral s.t.}\,\, \sigma_1(C_{\epsilon})P(C_{\epsilon})=C  \},
\end{equation*}  
is given by a rectangle.
\end{conjecture}

\bigskip\noindent

{\bf Acknowledgements}: 
The authors want to thank the reviewer for very good suggestions leading to an improvement of the paper.
The authors are also grateful to B. Bogosel for providing us the values of $\sigma_1(T_1)$ and $\sigma_1(T_2)$ with high numerical accuracy. This work was partially supported by the project ANR-18-CE40-0013 SHAPO financed by the French Agence Nationale de la Recherche (ANR).

\bibliographystyle{abbrv}
\bibliography{Ref3}

\begin{thebibliography}{10}

\bibitem{AH11}
P.~R.~S. Antunes and A.~Henrot.
\newblock On the range of the first two {D}irichlet and {N}eumann eigenvalues
  of the {L}aplacian.
\newblock {\em Proc. R. Soc. Lond. Ser. A Math. Phys. Eng. Sci.},
  467(2130):1577--1603, 2011.

\bibitem{BB99}
R.~Ba\~{n}uelos and K.~Burdzy.
\newblock On the ``hot spots'' conjecture of {J}. {R}auch.
\newblock {\em J. Funct. Anal.}, 164(1):1--33, 1999.

\bibitem{Bog17}
B.~Bogosel.
\newblock The {S}teklov spectrum on moving domains.
\newblock {\em Appl. Math. Optim.}, 75(1):1--25, 2017.

\bibitem{B29}
T.~Bonnesen.
\newblock Les probl\'{e}mes des isop\'{e}rim\'{e}tres et de is\`{e}phipanes.
\newblock {\em Gauthier-Villars,Paris}, 1929.

\bibitem{Bos86}
M.-H. Bossel.
\newblock Membranes \'elastiquement li\'ees: extension du th\'eor\`eme de
  {R}ayleigh-{F}aber-{K}rahn et de l'in\'egalit\'e de {C}heeger.
\newblock {\em C. R. Acad. Sci. Paris S\'er. I Math.}, 302(1):47--50, 1986.

\bibitem{BCL21}
B.~Brandolini, F.~Chiacchio, and J.~J. Langford.
\newblock Eigenvalue estimates for p-laplace problems on domains expressed in
  fermi coordinates, 2021, Arxiv arXiv:2106.13903.

\bibitem{BBFi}
D.~Bucur, G.~Buttazzo, and I.~Figueiredo.
\newblock On the attainable eigenvalues of the {L}aplace operator.
\newblock {\em SIAM J. Math. Anal.}, 30(3):527--536, 1999.

\bibitem{BGT20}
D.~Bucur, A.~Giacomini, and P.~Trebeschi.
\newblock {$L^\infty$} bounds of {S}teklov eigenfunctions and spectrum
  stability under domain variation.
\newblock {\em J. Differential Equations}, 269(12):11461--11491, 2020.

\bibitem{BHM21}
D.~Bucur, A.~Henrot, and M.~Michetti.
\newblock Asymptotic behaviour of the {S}teklov spectrum on dumbbell domains.
\newblock {\em Comm. Partial Differential Equations}, 46(2):362--393, 2021.

\bibitem{BN20}
D.~Bucur and M.~Nahon.
\newblock Stability and instability issues of the {W}einstock inequality.
\newblock {\em Trans. Amer. Math. Soc.}, 374(3):2201--2223, 2021.

\bibitem{Cheng75}
S.~Y. Cheng.
\newblock Eigenvalue comparison theorems and its geometric applications.
\newblock {\em Math. Z.}, 143(3):289--297, 1975.

\bibitem{EgKo}
Y.~{Egorov} and V.~{Kondratiev}.
\newblock {\em {On spectral theory of elliptic operators}}, volume~89.
\newblock Basel: Birkh\"auser, 1996.

\bibitem{HF21}
I.~Ftouhi and A.~Henrot.
\newblock The diagram $(\lambda_1,\mu_1$).
\newblock {\em Mathematical Reports}, 24 (74)(1-2), 2022.

\bibitem{GHL20}
A.~Girouard, A.~Henrot, and J.~Lagac\'{e}.
\newblock From {S}teklov to {N}eumann via homogenisation.
\newblock {\em Arch. Ration. Mech. Anal.}, 239(2):981--1023, 2021.

\bibitem{GKL20}
A.~{Girouard}, M.~{Karpukhin}, and J.~{Lagac\'e}.
\newblock {Continuity of eigenvalues and shape optimisation for Laplace and
  Steklov problems}.
\newblock {\em {Geom. Funct. Anal.}}, 31(3):513--561, 2021.

\bibitem{H85}
R.~R. Hall.
\newblock A class of isoperimetric inequalities.
\newblock {\em J. Analyse Math.}, 45:169--180, 1985.

\bibitem{HS20}
A.~Hassannezhad and A.~Siffert.
\newblock A note on {K}uttler-{S}igillito's inequalities.
\newblock {\em Ann. Math. Qu\'{e}.}, 44(1):125--147, 2020.

\bibitem{H06}
A.~Henrot.
\newblock {\em Extremum problems for eigenvalues of elliptic operators}.
\newblock Frontiers in Mathematics. Birkh\"{a}user Verlag, Basel, 2006.

\bibitem{HLL21}
A.~Henrot, A.~Lemenant, and I.~Lucardesi.
\newblock Maximizing $p^2 \mu_1$ among symmetric convex sets, 2022, preprint.

\bibitem{HPb}
A.~Henrot and M.~Pierre.
\newblock {\em Shape variation and optimization}, volume~28 of {\em EMS Tracts
  in Mathematics}.
\newblock European Mathematical Society (EMS), Z\"{u}rich, 2018.
\newblock A geometrical analysis.

\bibitem{H00}
M.~A. Hern\'{a}ndez~Cifre.
\newblock Is there a planar convex set with given width, diameter, and
  inradius?
\newblock {\em Amer. Math. Monthly}, 107(10):893--900, 2000.

\bibitem{JM92}
S.~Jimbo and Y.~Morita.
\newblock Remarks on the behavior of certain eigenvalues on a singularly
  perturbed domain with several thin channels.
\newblock {\em Comm. Partial Differential Equations}, 17(3-4):523--552, 1992.

\bibitem{KLR06}
B.~Kawohl and T.~Lachand-Robert.
\newblock Characterization of {C}heeger sets for convex subsets of the plane.
\newblock {\em Pacific J. Math.}, 225(1):103--118, 2006.

\bibitem{KT19}
D.~{Krej\v{c}i\v{r}\'{\i}k} and M.~{Tu\v{s}ek}.
\newblock {Location of hot spots in thin curved strips}.
\newblock {\em {J. Differ. Equations}}, 266(6):2953--2977, 2019.

\bibitem{Ku1923}
T.~Kubota.
\newblock Einige ungleischheitsbezichungen uber eilinien und eiflachen,.
\newblock {\em Sci. Rep. of the Tohoku Univ. Ser.}, 12(1):45--65, 1923.

\bibitem{KS68}
J.~R. Kuttler and V.~G. Sigillito.
\newblock Inequalities for membrane and {S}tekloff eigenvalues.
\newblock {\em J. Math. Anal. Appl.}, 23:148--160, 1968.

\bibitem{LP15}
P.~D. Lamberti and L.~Provenzano.
\newblock Viewing the {S}teklov eigenvalues of the {L}aplace operator as
  critical {N}eumann eigenvalues.
\newblock In {\em Current trends in analysis and its applications}, Trends
  Math., pages 171--178. Birkh\"{a}user/Springer, Cham, 2015.

\bibitem{LZ19}
I.~Lucardesi and D.~Zucco.
\newblock On {B}laschke-{S}antal\'{o} diagrams for the torsional rigidity and
  the first {D}irichlet eigenvalue.
\newblock {\em Annali di Matematica Pura ed Applicata}, 201:175--201, 2022.

\bibitem{PW60}
L.~E. Payne and H.~F. Weinberger.
\newblock An optimal {P}oincar\'{e} inequality for convex domains.
\newblock {\em Arch. Rational Mech. Anal.}, 5:286--292 (1960), 1960.

\bibitem{S60}
H.~Sachs.
\newblock Ungleichungen f\"{u}r {U}mfang, {F}l\"{a}cheninhalt und
  {T}r\"{a}gheitsmoment konvexer {K}urven.
\newblock {\em Acta Math. Acad. Sci. Hungar.}, 11:103--115, 1960.

\bibitem{SA00}
P.~R. Scott and P.~W. Awyong.
\newblock Inequalities for convex sets.
\newblock {\em JIPAM. J. Inequal. Pure Appl. Math.}, 1(1):Article 6, 6, 2000.

\bibitem{Szeg54}
G.~Szeg\"{o}.
\newblock Inequalities for certain eigenvalues of a membrane of given area.
\newblock {\em J. Rational Mech. Anal.}, 3:343--356, 1954.

\bibitem{T65}
B.~A. Troesch.
\newblock An isoperimetric sloshing problem.
\newblock {\em Comm. Pure Appl. Math.}, 18:319--338, 1965.

\bibitem{BBP19}
M.~{van den Berg}, G.~{Buttazzo}, and A.~{Pratelli}.
\newblock {On relations between principal eigenvalue and torsional rigidity}.
\newblock {\em {Commun. Contemp. Math.}}, 23(8):28, 2021.
\newblock Id/No 2050093.

\bibitem{W85}
H.~T. Wang.
\newblock Convex functions and {F}ourier coefficients.
\newblock {\em Proc. Amer. Math. Soc.}, 94(4):641--646, 1985.

\bibitem{W56}
H.~F. Weinberger.
\newblock An isoperimetric inequality for the {$N$}-dimensional free membrane
  problem.
\newblock {\em J. Rational Mech. Anal.}, 5:633--636, 1956.

\bibitem{Wei54}
R.~Weinstock.
\newblock Inequalities for a classical eigenvalue problem.
\newblock {\em J. Rational Mech. Anal.}, 3:745--753, 1954.

\end{thebibliography}

\end{document}